\documentclass{elsarticle}

\usepackage{epsfig,psfrag}
\usepackage{amsmath}
\usepackage{amssymb}
\usepackage{latexsym,pifont,color,comment}

\newcommand{\reals}{\mathbb{R}}

\newcommand{\Tm}{{\mathcal T}}

\newcommand\x{{\bf x}}

\newcommand{\ft}{\tilde{f}}
\newcommand{\xt}{\tilde{x}}
\newcommand{\vt}{\tilde{v}}

\newcommand{\Et}{\tilde{E}}

\newcommand{\rhot}{\tilde{\rho}}

\def\x{{\bf x}}

\def\ft{\tilde{f}}
\def\tti{\tilde{t}}
\def\xt{\tilde{x}}
\def\vt{\tilde{v}}
\def\Et{\tilde{E}}
\def\rhot{\tilde{\rho}}
\def\Flux{{\mathcal F}}

\newcommand{\M}[1]{M_{#1}}

\usepackage{algorithm,algorithmic}
\usepackage{amsthm}
\newtheorem{theorem}{Theorem}[section]
\newtheorem{remark}{Remark}

\journal{J. Comput. Appl. Math.}

\begin{document}

\begin{frontmatter}

%\title{Asymptotic-Preserving Quadrature-Based Moment-Closure Methods
%for the Vlasov-Poisson-\\ Fokker-Planck System in the High-Field Limit}

\title{A Class of Quadrature-Based Moment-Closure Methods
with Application to the Vlasov-Poisson-Fokker-Planck System in the High-Field Limit}

\author[author1]{Yongtao Cheng}
\ead{cheng@math.wisc.edu}

\author[author2]{James A. Rossmanith\fnref{labc}}
\ead{rossmani@iastate.edu}

\address[author1]{University of Wisconsin,
Deparment of Mathematics, 480 Lincoln Drive,
Madison, WI 53706, USA}

\address[author2]{Iowa State University, Department of Mathematics,
396 Carver Hall, Ames, IA 50011, USA}

\fntext[labc]{Corresponding author}

\begin{abstract}
Quadrature-based moment-closure methods are a class of
approximations that replace high-dimensional kinetic
descriptions with lower-dimensional 
fluid models. In this work we investigate some of the
properties of a sub-class of these methods based on
bi-delta, bi-Gaussian, and bi-B-spline representations.
We develop a high-order discontinuous Galerkin (DG) scheme
to solve the resulting fluid systems. Finally, via this
high-order DG scheme and Strang operator splitting to handle the
collision term, we simulate the fluid-closure models in the
context of the Vlasov-Poisson-Fokker-Planck system in the
high-field limit. We demonstrate numerically that the proposed scheme
is asymptotic-preserving in the high-field limit.
\end{abstract}


\end{frontmatter}

\section{Introduction}
\label{sec:intro}
The focus of this work is on 1D fluid models of plasma with a class of fluid-closure
approximations known as {\it quadrature-based moment-closures}.
In particular, in this work we are interested in applying these fluid-closure
approximations to the one-dimensional form the Vlasov-Poisson-Fokker-Planck (VPFP) equations
(see for example Bonilla et al. \cite{article:BoCaSo97}):
\begin{gather}
 \label{eqn:VPFP_dim}
\ft_{,\tti} + \vt \ft_{,\xt} - \frac{e}{m} \Et \ft_{,\vt}  = 
	\mu \left( \vt \ft + \frac{k_B \Theta}{m} \ft_{,\vt}  \right)_{,\vt}, \quad
\Et_{,\xt} = \frac{e}{m \epsilon_0} \left( \rhot_0 - \rhot \right),
\end{gather}
where $\tti \in \reals$ is time, $\xt \in \reals$ is the spatial coordinate,
$\vt \in \reals$ is the velocity, $\ft(\tti,\xt,\vt)$ is the probability density function
for electrons, $\Et(\tti,\xt)$ is the electric field, and $\rhot = \int m \tilde{f} d\vt$ is the 
electron mass density. 
In order to avoid confusion between indices and partial derivatives, we adopt the convention from general relativity of using a comma to denote partial derivatives.
The parameters in this equation are
the elementary charge $e$, the electron mass $m$, the Boltzmann constant $k_B$,
the temperature of the equilibrium state $\Theta$, the stationary background
ion mass density $\rhot_0(\xt)$, and the collision frequency $\mu$.
In the above expression $\tilde{\cdot}$ is used to denote dimensional dependent and
independent variables (i.e., these decorations will be removed after
non-dimensionalization).  These equations describe the dynamics of electrons (as represented
by the PDF $\ft(\tti,\xt,\vt)$) that evolve via Coulomb interactions and collisions in the form
of a Fokker-Planck drift-diffusion operator. The Fokker-Planck operator tries to drive the
system to a thermodynamic equilibrium with constant temperature $\Theta$.

\subsection{The Vlasov-Poisson-Fokker-Planck system in the high-field limit}
Fluid-closure methods as described in this work will generally not accurately
approximate solutions of \eqref{eqn:VPFP_dim} in the collisionless limit,
$\mu \rightarrow 0$. Instead, we focus here on the 
high-collision limit; and in particular, the {\it high-field limit}, which describes
the long-time, large-scale, high-collisional, and large electric field limit
of the VPFP system.
The high-field limit of the VPFP has been considered by many authors both
theoretically and numerical, including by Arnold et al. \cite{article:ArCaGaSh01},
Bonilla and Soler \cite{article:BoSo01},
Cercignani et al. \cite{CeGaLe97}, Nieto et al. \cite{article:nieto01}, and
Wang and Jin \cite{article:WaJi11}.

In order to derive the high-field limit of the VPFP system we introduce
a non-dimensionalization via the characteristic
scaling constants: $T$ (time), $L$ (length), $E_0$ (electric field),  and $N$ (number density),
such that
\begin{gather*}
 \tti= T t, \quad \xt = L x, \quad \vt = L T^{-1} v , \quad  \Et = E_0 E, \quad
 \rhot = m N \rho, \quad \ft = N T L^{-1} f.
 \end{gather*}
This reduces the VPFP system \eqref{eqn:VPFP_dim} to
\begin{gather*}
f_{,t} + v f_{,x} - \mu T \left( \frac{e E_0 T}{\mu m L} \right) E f_{,v} = 
  \mu T \left( vf + \left( \frac{k_B T^2 \Theta}{mL^2} \right) f_{,v} \right)_{,v}, \quad
E_{,x} = \left( \frac{eLN}{\epsilon_0 E_0} \right) \left( \rho_0 - \rho \right).
\end{gather*}
We define the dimensionless parameter $\varepsilon = \left(\mu T \right)^{-1}$
and choose $T$, $L$, and $E_0$ as follows:
\begin{equation}
\label{eqn:scales}
  T = \frac{m \epsilon_0 \mu}{N e^2}, \quad
  L = \sqrt{mk_B \Theta} \left( \frac{\epsilon_0 \mu}{N e^2} \right),
  \quad E_0 = \frac{\mu}{e} \sqrt{m k_B \Theta}.
\end{equation}
With these choices we arrive at the following non-dimensional Vlasov-Poisson Fokker-Planck (VPFP)
system:
\begin{equation}
\label{eqn:VPFP}
f_{,t} + v f_{,x} = \frac{1}{\varepsilon} \left( F \left( F^{-1} f  \right)_{,v} \right)_{,v}, \quad
E_{,x} = \rho_0 - \rho,
\end{equation}
where $\rho = \int f dv$ and $F(t,x,v)$ is the isothermal equilibrium distribution:
\begin{equation}
\label{eqn:equilibrium}
   F(t,x,v) = \frac{\rho(t,x)}{\sqrt{2\pi}} \exp\left(-\frac{1}{2} \left( v + E(t,x) \right)^2\right).
\end{equation}
%\begin{equation}
%f_{,t} + v f_{,x} - \frac{1}{\varepsilon} E f_{,v} = 
%	\frac{1}{\varepsilon} \left( v f + f_{,v}  \right)_{,v}, \quad
%E_{,x} = \rho_0 - \rho,
%\end{equation}
%where $\rho = \int f dv$.
%A more compact way to write this system is to group
%all of the terms of order $\varepsilon^{-1}$ on the right-hand
%side:

The so-called {\it high-field limit} is when $\varepsilon \rightarrow 0^{+}$, which, under
the choices of the characteristic time, length, and electric field chosen in \eqref{eqn:scales}, describes
long time ($T\rightarrow \infty$), large-scale ($L\rightarrow \infty$), and large electric field
($E_0 \rightarrow \infty$) dynamics of the
VPFP system. In particular, Nieto et al. \cite{article:nieto01} proved
 that as $\varepsilon \rightarrow 0^+$, the solution of the VPFP system \eqref{eqn:VPFP} 
converges to the equilibrium distribution \eqref{eqn:equilibrium}, where (to leading
order in $\varepsilon$) the mass density, $\rho(t,x)$, and the electric field, $E(t,x)$, satisfy the 
following non-local advection equation:
\begin{align}
\label{eqn:rho_eq}
\rho_{,t} - \left( \rho E \right)_{,x} = 0, \quad E_{,x} = \rho_0 - \rho.
\end{align}
The non-local nature of this advection equation comes from the fact that the electric
field, which serves as the advection velocity field,
is affected globally by local modifications in the charge density. 

\subsection{Scope of this work}
Recently, Wang and Jin \cite{article:WaJi11}
developed an {\it asymptotic-preserving} scheme for the VPFP equation, 
where they modified a fully kinetic solver for
VPFP so that it remains asymptotic preserving in the high-field limit
$\varepsilon \rightarrow 0^+$. The approach has the nice property
that it can be applied for {\it any} value of
$\varepsilon>0$. 
The problem with the Wang and Jin \cite{article:WaJi11} scheme is
that if one is really interested in regimes where $\varepsilon$
is relatively small (i.e., {\it near} thermodynamic equilibrium), then their
approach is computationally expensive (i.e., requires solving a PDE in
2D rather than 1D).  

The purpose of this work is to consider fluid models for the
VPFP system that not only have the ability to capture the equilibrium dynamics
of VPFP (i.e., equation \eqref{eqn:rho_eq}), but also accurately model near-equilibrium 
dynamics. The  scope of the current work is to twofold:
\begin{enumerate}
\item We describe and investigate properties of 
	two approaches in the quadrature-based moment-closure
	framework as developed by Fox \cite{article:fox09} (bi-delta
	distribution functions) and Chalons, Fox, and Massot \cite{article:Chalons10}
	(bi-Gaussian distribution functions), and describe
	a modification of these based on bi-B-spline distributions.
	This is described in \S \ref{sec:quad-based-moment-closure}.
\item We then take this class of quadrature-based moment-closure
approaches and, via a high-order discontinuous Galerkin scheme
with Strang operator splitting for the collision operator, approximately
solve the VPFP system in the high-field limit. The numerical
method is developed in \S \ref{sec:numerical-method} and applied to two test problems
from Wang and Jin \cite{article:WaJi11} in \S \ref{sec:numerical-examples}.
\end{enumerate}

%The moments satisfy conservation laws of the form:
%\begin{align}
 % \M{\ell,t} + \M{\ell+1,x} = \frac{1}{\varepsilon} \left\{ \ell (\ell-1) \M{\ell-2} 
  % - \ell E \M{\ell-1} - \ell \M{\ell} \right\}.
%\end{align}

 \section{Quadrature-based moment-closure methods}
 \label{sec:quad-based-moment-closure}
In this section we describe three variants of the quadrature-based
moment-closure approach. In \S \ref{sec:delta} we 
describe the simplest of these: quadrature based on delta distributions.
Delta distributions have been used in
 several previous works including in gas dynamics
applications in Fox \cite{article:fox09} and Yan and Fox \cite{article:yuan11},
multivalued solutions of Euler-Poisson \cite{article:Li04}, and
multiphase solutions of the semiclassical limit of the
Schr\"odinger equation \cite{article:Go03,article:Jin03}.
In \S \ref{sec:gauss} we describe a generalization of this approach using
Gaussian distributions developed by Chalons, Fox, and Massot \cite{article:Chalons10}.
Finally in \S \ref{sec:hat} we propose a modification of the Gaussian distribution
approach that uses compactly supported B-splines.
%In all of these sections we focus on quadrature using only two points; in \S \ref{sec:higher} we briefly 
%explain how to extend this to higher-order quadrature rules.

\subsection{Quadrature using delta distributions}
\label{sec:delta}
In the quadrature-based moment-closure approach using
delta distributions we assume a PDF that is a sum of
delta functions with unknown weights and positions (we show
here the simplest version of this approach using only two quadrature points):
\begin{align}
\label{eqn:delta}
  \overline{f}(t,x,v) = \omega_1 \delta \left( v-\mu_1 \right) + \omega_2 \delta \left( v-\mu_2 \right),
\end{align}
where the parameters $\omega_1$, $\mu_1$, $\omega_2$, and $\mu_2$ are all functions
of $t$ and $x$. This approach is reminiscent of other discrete velocity models such
as the Broadwell model \cite{article:broad64a,article:broad64b}; however, 
a key difference is that the discrete velocities, $\mu_1$ and $\mu_2$, are potentially
different at each point in space and time.

The first four moments of \eqref{eqn:delta} are
\begin{alignat}{2}
\label{eqn:mom_delta1}
  \M{0} &= \rho &&= \omega_1 \mu^0_1 + \omega_2 \mu^0_2, \\
  \label{eqn:mom_delta2}
  \M{1} &= \rho u &&= \omega_1 \mu^1_1 + \omega_2 \mu^1_2, \\
  \label{eqn:mom_delta3}
  \M{2} &= \rho u^2 + p &&= \omega_1 \mu_1^2 + \omega_2 \mu_2^2, \\
  \label{eqn:mom_delta4}
  \M{3} &= \rho u^3 + 3 pu + q && = \omega_1 \mu_1^3 + \omega_2 \mu_2^3.
\end{alignat}
If we assume that $\rho>0$ and $p>0$, then the above relationship between
the moments $(\M{0}, \M{1}, \M{2}, \M{3})$ and the parameters $(\mu_1, \mu_2, \omega_1, \omega_2)$
is one-to-one (see discussion below).  In the absence of collisions, these moments satisfy
equations of the form:
\begin{align}
  \M{\ell,t} + \M{\ell+1,x} = 0
\end{align}
for $\ell=0,1,2,3$. The moment-closure comes from forcing $M_4$ to come from
\eqref{eqn:delta} (rather than letting it be an independent quantity):
\begin{equation}
\label{eqn:m4_delta}
	M_4 \leftarrow \overline{M}_4 \equiv \int_{-\infty}^{\infty} v^4 \overline{f} \, dv = 
		\omega_1 \mu_1^4 + \omega_2 \mu_2^4.
\end{equation}
Therefore, the {\bf moment-closure problem} is reduced to the following:
given the first four moments of the system, $(\M{0}, \M{1}, \M{2}, \M{3})$,
solve system \eqref{eqn:mom_delta1}--\eqref{eqn:mom_delta4} to obtain
$(\mu_1, \mu_2, \omega_1, \omega_2)$, then use these parameters
to calculate $\overline{M}_4$ via \eqref{eqn:m4_delta}.

Solving system  \eqref{eqn:mom_delta1}--\eqref{eqn:mom_delta4} 
is equivalent to finding the quadrature points and weights
for the following weighted Gaussian quadrature rule:
\begin{equation}
   \int_{-\infty}^{\infty} g(v) \, w(v) \, dv \, \approx \,  {\omega}_1 \, g \left( {\mu}_1 \right) 
   	+ {\omega}_2 \, g \left( {\mu}_2 \right)
\end{equation}
with weight function $w(v)$ that satisfies
\begin{equation}
 \int_{-\infty}^{\infty} v^k \, w(v) \, dv = M_k \quad \text{for} \quad k=0,1,2,3.
\end{equation}
If we attempt to make this quadrature
rule exact with $g(v) = 1$, $v$, $v^2$,  and $v^3$, we arrive at equations \eqref{eqn:mom_delta1}--\eqref{eqn:mom_delta4}.
To find the correct Gaussian quadrature rule, we invoke results from classical numerical analysis
(e.g., see pages 220-225 of Burden and Faires \cite{article:BuFa05})
and look for polynomials of degree up to two that are orthogonal in the following
weighted inner product:
\begin{equation}
\label{eqn:weight_inner_product}
	\langle g , h \rangle_{w} := \int_{-\infty}^{\infty} g(v) \, h(v) \, w(v) \, dv.
\end{equation}
Such polynomials are easily obtained by starting with monomials in $v$
and applying Gram-Schmidt orthogonalization with respect to \eqref{eqn:weight_inner_product}:
\begin{align}
  {\psi}^{(0)}(v) &= 1, \\
  {\psi}^{(1)}(v) &= v - u, \\ 
  {\psi}^{(2)}(v) &= 3 \rho p  v^2 - \left(6 \rho  p  u + 3 \rho q \right) v + \left( 3 \rho p u^2 - 3 p^2 + 3 
   u \rho  q \right).
\end{align}
The quadrature points $\mu_1$ and $\mu_2$ are the two  distinct real roots of ${\psi}^{(2)}(v)$:
\begin{align}
  {\mu}_1, \, {\mu}_2 = u + \frac{q}{2p} \mp \sqrt{\frac{p}{\rho} + \left(\frac{q}{2p}\right)^2}.
\end{align}
Once the quadrature points are known, the corresponding quadrature weights can easily
be obtained by solving equations \eqref{eqn:mom_delta1} and  \eqref{eqn:mom_delta2} for
the weights:
\begin{align}
  {\omega}_1, \, \omega_2 = \frac{\rho}{2} \pm \frac{\rho^2  q}{2 \sqrt{\rho^2  q^2 + 4 \rho  p^3}}.
\end{align}

\subsubsection{Flux form of the bi-delta system}
Once $\mu_1$, $\mu_2$, $\omega_1$, and $\omega_2$ have been computed, we
can evaluate the moment closure: replace the true $\M{4}$ with the following:
\begin{equation}
  \overline{M}_4 = \mu_1 \, \omega_1^4 +  \mu_2 \, \omega_2^4 = 
  	 \rho u^4 + 6 p u^2 + 4 u q + \frac{q^2}{p}  + \frac{p^2}{\rho}.
\end{equation}
From this we can write down the fully closed system:
\begin{equation}
	\label{eqn:system}
U_{,t} + F(U)_{,x} = 0, 
\end{equation}
where
\begin{equation}
\label{eqn:delta_U}
U = \left[ \rho, \, \rho u, \, \rho u^2 + p, \,  \rho u^3 + 3 p  u + q\right]^T
\end{equation}
 is the vector of conserved variables and 
\begin{align}
\label{eqn:delta_F}
F(U) =  \left[ \rho u, \, \rho u^2 + p, \, \rho u^3 + 3 p  u + q, \,
\rho u^4 + 6 p u^2 + 4 u q + \frac{q^2}{p}  + \frac{p^2}{\rho} \right]^T
\end{align}
is the flux function.

\subsubsection{Hyperbolic structure of the bi-delta system}
One can show that system \eqref{eqn:system} with \eqref{eqn:delta_U}--\eqref{eqn:delta_F}, assuming that $\rho>0$ and $p>0$, is a {\it weakly hyperbolic} system of PDEs.
The eigenvalues of the flux Jacobian, $F_{,U}$, are
\begin{align}
 \lambda^{(1)} = \lambda^{(2)} = \mu_1 \quad \text{and} \quad
  \lambda^{(3)} = \lambda^{(4)} = \mu_2.
\end{align}
The two linearly independent eigenvectors are
\begin{align}
   r^{(1)} = r^{(2)} = \bigl[ 1, \mu_1, \mu_1^2, \mu_1^3 \bigr]^T \quad \text{and} \quad
   r^{(3)} = r^{(4)} = \bigl[ 1, \mu_2, \mu_2^2, \mu_2^3 \bigr]^T.
\end{align}
One can show that the simple waves associated with each of these eigenvectors are
linearly degenerate:
\begin{equation}
  \nabla_U \lambda^{(k)} \cdot r^{(k)} = 0 \quad \text{for} \quad k=1,2,3,4,
\end{equation}
where $\nabla_U$ is the gradient with respect to the conserved
variables  \eqref{eqn:delta_U}.
For a detailed analysis of this system see Chalons, Kah, and Massot \cite{article:ChKaMa12}.

\subsection{Quadrature using Gaussian distributions}
\label{sec:gauss}
There are two main difficulties with moment-closure based on quadrature
via delta distributions: (1) the resulting system is only weakly hyperbolic,
which means that delta shocks \cite{article:Bou94,article:ChKaMa12,article:ChLiu03} are generically present in the system,
and (2) a large number of delta functions may be required to get good agreement
with smooth distributions with large support. 
Yuan and Fox \cite{article:yuan11} developed
an adaptive Gaussian quadrature strategy to help overcome this problem. 
A possible alternative to using large number of quadrature points  was developed Chalons et al.
\cite{article:Chalons10} who introduced a quadrature moment-closure
based on replacing the Dirac delta functions with Gaussian distribution functions.
In particular, in order to simplify the moment inversion equations, each Gaussian
is assumed to have the same standard deviation. In the case of 
two Gaussian distributions this results in an assumed distribution function of the form:
 \begin{align}
 \label{eqn:bigaussf}
  \overline{f}(t,x,v) = \frac{\omega_1}{\sqrt{2\pi \sigma}} \exp\left( -\frac{\left(v-\mu_1 \right)^2}{2\sigma} \right)
  	+ \frac{\omega_2}{\sqrt{2\pi \sigma}} \exp\left( -\frac{\left(v-\mu_2 \right)^2}{2\sigma} \right),
\end{align}
where $\sigma$ is the width. In the limit as $\sigma\rightarrow 0$ we recover the 
Dirac delta distribution moment-closure method.
In this section we describe the moment-inversion algorithm needed 
to convert between moments and the parameters in representation \eqref{eqn:bigaussf}.
Furthermore, we briefly discuss the hyperbolic structure of the resulting evolution
equations for the moments of \eqref{eqn:bigaussf}.

There are now five free parameters: $\left(\mu_1, \mu_2, \omega_1, \omega_2, \sigma \right)$, which can be obtained by solving the following moment-inversion problem on the first five moments:
\begin{align}
\M{0} &= \omega_1 \mu^0_1 + \omega_2 \mu^0_2, \\
  \M{1} &= \omega_1 \mu^1_1 + \omega_2 \mu^1_2, \\
  \M{2} &= \omega_1 \mu_1^2 + \omega_2 \mu_2^2 + \sigma \left(\omega_1 \mu^0_1 + \omega_2 \mu^0_2\right), \\
  \M{3} & = \omega_1 \mu_1^3 + \omega_2 \mu_2^3 + 3 \sigma \left(\omega_1 \mu^1_1 + \omega_2 \mu^1_2\right), \\
  \M{4} & = \omega_1 \mu_1^4 + \omega_2 \mu_2^4 + 6\sigma \left( \omega_1 \mu_1^2 + \omega_2 \mu_2^2 \right)
  	+  3 \sigma^2  \left( \omega_1 \mu_1^0 + \omega_2 \mu_2^0 \right).
\end{align}
We write this system in terms of primitive variables: 
\begin{align}
\label{eqn:MG1}
    \omega_1 \mu^0_1 + \omega_2 \mu^0_2 \, &=  \, \rho, \\
\label{eqn:MG2}
    \omega_1 \mu^1_1 + \omega_2 \mu^1_2  \, &=  \, \rho u, \\
\label{eqn:MG3}
    \omega_1 \mu_1^2 + \omega_2 \mu_2^2  \, &=  \,  \rho u^2 + \alpha p, \\
\label{eqn:MG4}
    \omega_1 \mu_1^3 + \omega_2 \mu_2^3 \,  &=  \, \rho u^3 + 3 \alpha pu + q , \\
\label{eqn:MG5}
   \omega_1 \mu_1^4 + \omega_2 \mu_2^4  \, &=  \, \rho u^4 + 6 \alpha p u^2 + 4 q u + r 
  + \frac{3p^2 (\alpha^2-1)}{\rho},
\end{align}
where we have introduced the parameter $\alpha$:
\begin{gather}
   \sigma = \frac{p}{\rho} \left( 1 - \alpha \right).
\end{gather}
We note that the maximum allowable value of $\alpha$ is clearly 1 (otherwise $\sigma$ would
be negative). In fact if $\alpha=1$, then the bi-Gaussian and the bi-delta function
representations are equivalent. What is perhaps less obvious is that the
minimum allowed value of $\alpha$ is zero. In fact, as $\alpha \rightarrow 0$,
the bi-Gaussian representation collapses into a single Gaussian distribution
with width given by the temperature $p/\rho$.
Physically reasonable conditions on the primitive variables
guarantee that the $\alpha$ that comes from 
solving the moment equations 
\eqref{eqn:MG1}--\eqref{eqn:MG5} above satisfies $\alpha \in [0,1]$ (see Theorem \ref{thm:mom-realz}).
%The exact formula for computing the value of $\alpha$ is given in Algorithm \ref{alg:alpha}.

\begin{theorem}[Moment-realizability condition, modified from Chalons et al. \cite{article:Chalons10}]
\label{thm:mom-realz}
Assume that the primitive variables satisfy the following conditions:
\begin{itemize}
\item Positive density: \quad $0< \rho$,
\item Positive pressure: \quad $0<p$,
\item Lower bound on $r$: \quad $\frac{p^3 + \rho q^2}{\rho p} \le r$,
\item If $q=0$, bound on $r$: \quad $\frac{p^2}{\rho} \le r \le \frac{3p^2}{\rho}$.
\end{itemize}
\begin{enumerate}
\item If $q\ne0$ then there exists a unique 
$\alpha \in (0,1]$ that satisfies the following cubic polynomial:
\begin{equation}
\label{eqn:poly-alpha}
{\mathcal P}(\alpha) = 2 p^3 \alpha^3 + \left(\rho r
 - 3 p^2 \right)  p \alpha - \rho q^2 = 0.
\end{equation}
Furthermore, from this $\alpha$ we can uniquely obtain 
the quadrature abscissas and weights in order to 
exactly solve system \eqref{eqn:MG1}--\eqref{eqn:MG5}:
\begin{align}
\label{eqn:mus}
\mu_1 , \mu_2 &=  u + \frac{q}{2 p \alpha}
  \mp {\sqrt{\frac{p\alpha}{\rho} + \left( \frac{q}{2p\alpha} \right)^2 }}, \\
  \label{eqn:omegas}
{\omega}_1, \, \omega_2 &= \frac{\rho}{2} \pm \frac{\rho^2  q}{2 \sqrt{\rho^2  q^2 + 4 \rho  p^3 \alpha^3}}.
\end{align}

\item If $q=0$ and $\frac{p^2}{\rho} \le r < \frac{3p^2}{\rho}$  then there exists a unique 
$\alpha \in (0,1]$ such that
\begin{equation}
\label{eqn:alpha_q0}
   \alpha = \sqrt{\frac{3 p^2 - \rho r}{2p^2}}.
\end{equation}
Furthermore, in this case
the quadrature abscissas \eqref{eqn:mus} and weights \eqref{eqn:omegas}
 are again the unique solutions of system \eqref{eqn:MG1}--\eqref{eqn:MG5}.

\item If $q=0$ and $r = \frac{3p^2}{\rho}$, then $\alpha=0$. This case corresponds to a single Gaussian
 distribution.  In this case we lose uniqueness of the quadrature abscissas and weights, but without loss of
 generality we can take
\begin{align}
\label{eqn:mus_zero}
\mu_1 , \mu_2 &= u, \\
\label{eqn:omegas_zero}
{\omega}_1, \, \omega_2 &= \frac{\rho}{2},
\end{align}
and still exactly solve system \eqref{eqn:MG1}--\eqref{eqn:MG5}.

\end{enumerate}
\end{theorem}

\begin{proof}
We take each point in turn.
\begin{enumerate}
\item Let us momentarily assume that $\alpha$ is known and that $\alpha \in (0,1]$. 
In this situation we are left with four unknowns: $(\mu_1, \mu_2, \omega_1, \omega_2)$, which
are determined by satisfying the first four moment equations: 
\eqref{eqn:MG1}--\eqref{eqn:MG4}. Just as in the case of quadrature-based moment-closures
using delta distributions, these equations can be 
solved by constructing a set of  polynomials that are mutually orthogonal
in the inner product \eqref{eqn:weight_inner_product}. The
weight function satisfies:
\begin{align}
\int_{-\infty}^{\infty} v^k \, w(v) \, dv = 
\begin{cases}
 \rho & \quad \text{if} \quad k=0, \\
  \rho u & \quad \text{if} \quad k=1, \\
  \rho u^2 + \alpha p & \quad \text{if} \quad k=2, \\
    \rho u^3 + 3 \alpha pu + q & \quad \text{if} \quad k=3.
\end{cases}
\end{align}
Up to degree 2 these mutually orthogonal polynomials can be written as
\begin{align}
\psi^{(0)}(v) &= 1, \\
\psi^{(1)}(v) &= v - u, \\
\psi^{(2)}(v) &= v^2 - \left(2u + \frac{q}{p \alpha} \right) v
	 + \left( u^2 + \frac{q u}{p \alpha} - \frac{p \alpha}{\rho} \right).
 \end{align}
 It is easy to show that the two real roots of $\psi^{(2)}(v)$ are \eqref{eqn:mus}.
 
 \medskip
 
 The weights \eqref{eqn:omegas} are obtained by plugging \eqref{eqn:mus} into 
 \eqref{eqn:MG1} and \eqref{eqn:MG2} and solving the resulting $2\times2$ linear
system for $\omega_1$ and $\omega_2$.

\medskip

Finally, we must obtain a formula for $\alpha \in (0,1]$. 
This is achieved by plugging \eqref{eqn:mus} and \eqref{eqn:omegas}
into the final moment equation: \eqref{eqn:MG5}. After simplification
this yields the cubic polynomial equation given by \eqref{eqn:poly-alpha}.
We note that under our assumptions we find that
\[
	{\mathcal P}(0) = -\rho q^2 < 0 \quad \text{and} \quad {\mathcal P}(1)=\rho p \left( r - \frac{p^3+\rho q^2}{\rho p} \right) \ge 0.
\]
Therefore by continuity of ${\mathcal P}(\alpha)$ we are guaranteed that there exists at least one root in $(0,1]$. 
To establish that there is a unique root in $(0,1]$ we note that ${\mathcal P}(\alpha)$ is convex in $(0,1]$:
\[
	{\mathcal P}''(\alpha) = 12 p^3 \alpha > 0 \quad \text{in} \quad (0,1].
\]	
%The exact formula for computing the value of $\alpha$ is given in Algorithm \ref{alg:alpha}.

\item If $q=0$ and $\frac{p^2}{\rho} \le r < \frac{3p^2}{\rho}$ we note that \eqref{eqn:poly-alpha}
reduces to
\[
	{\mathcal P}(\alpha) = \alpha \left( 2 p^2 \alpha^2 + \left(\rho r
 - 3 p^2 \right)  \right) = 0.
\]
The unique root of ${\mathcal P}(\alpha)$ in $(0,1]$ is given by \eqref{eqn:alpha_q0}.
With this value of $\alpha$ we can again solve \eqref{eqn:MG1}--\eqref{eqn:MG5}
using the $\mu_1$, $\mu_2$, $\omega_1$, and $\omega_2$ given by
\eqref{eqn:mus} and \eqref{eqn:omegas}.

\item If $q=0$ and $r = \frac{3p^2}{\rho}$ the only solution of
\eqref{eqn:poly-alpha} is $\alpha=0$. 
In this case the moment equations \eqref{eqn:MG1}--\eqref{eqn:MG5} reduce to
\[
     M_k = \omega_1 \mu_1^k + \omega_2 \mu_2^k \quad \text{for} \quad k=0,1,2,3,4.
\]
This system has an infinite number of solutions of the form:
\[
      \mu_1 = \mu_2 = u \quad \text{and} \quad \omega_1 + \omega_2 = \rho.
\]
Without loss of generality we take \eqref{eqn:mus_zero} and \eqref{eqn:omegas_zero}.

\end{enumerate}
\end{proof}

\subsubsection{Flux form of the bi-Gaussian system}
After obtaining $\alpha$, $\mu_1$, $\mu_2$, $\omega_1$, and $\omega_2$, we
impose the following moment-closure:
\begin{equation}
\label{eqn:biG_M5}
  \overline{M}_5 = \mu_1 \, \omega_1^5 +  \mu_2 \, \omega_2^5
  + 10 \sigma M_3 - 15 \sigma^2 M_1
  = \rho u^5 + 10 p u^3 + \frac{15 p^2 u}{\rho} + \alpha \tilde{M}_5,
\end{equation}
where
\begin{align}
\tilde{M}_5 &= \left( \frac{\tilde{q}^3}{p^2} + \frac{5 \tilde{q}^2 u}{p} + 10 \tilde{q} u^2 + \frac{10 p \tilde{q}}{\rho} \right) - 
 \frac{2 p \alpha}{\rho} \left(4 \tilde{q} + 5 p u \right), \\
 \label{eqn:qtilde}
 \tilde{q} &= q/\alpha \quad \text{(we set $\tilde{q}=0$ if $\alpha=0$)}.
\end{align}
Finally, we write the closed bi-Gaussian quadrature-based moment-closure
system in the form \eqref{eqn:system}, where
\begin{align}
\label{eqn:biGU}
U &= 
\left[
\rho, \, \rho u, \, \rho u^2 + p, \,  \rho u^3 + 3 p  u + q, \, \rho u^4 + 6 p u^2 + 4 q u + r
\right]^T, \\
\label{eqn:biGF}
F(U) &= 
\left[
\rho u, \, \rho u^2 + p, \, \rho u^3 + 3 p  u + q, \,
\rho u^4 + 6 p u^2 + 4 q u + r, \, \overline{M}_5
\right]^T.
\end{align}

\subsubsection{Hyperbolic structure of the bi-Gaussian system}
The flux Jacobian of the bi-Gaussian system described above
can be written in the following form:
\begin{equation}
F_{,U} = 
\begin{bmatrix}
0 & 1 & 0 & 0 & 0 \\
0 & 0 & 1 & 0 & 0 \\
0 & 0 & 0 & 1 & 0 \\
0 & 0 & 0 & 0 & 1 \\
\overline{M}_{5, M_0} &
\overline{M}_{5, M_1} &
\overline{M}_{5, M_2} &
\overline{M}_{5, M_3} &
\overline{M}_{5, M_4}
\end{bmatrix},
\end{equation}
where $U=(M_0, M_1, M_2, M_3, M_4)^T$ and $\overline{M}_{5}$ is given by
\eqref{eqn:biG_M5}.
The eigenvalues and right eigenvectors of the flux Jacobian are of the form
\begin{equation}
\lambda^{(k)} = z_k \quad \text{and} \quad
 r^{(k)} = \left[ 1, \, z_k, \, z_k^2, \, z_k^3, \, z_k^4 \right]^T,
\end{equation}
and the left eigenvectors can be written as
\begin{equation}
 \ell^{(k)} =  
\frac{\left[ \,\prod\limits_{j=1}^5 z_j, \, -\sum\limits_{\substack{j=1 \\ \ell=j+1 \\ m=\ell+1}}^{3, 4, 5}
  z_j z_{\ell} z_m, \, \sum\limits^{4,5}_{\substack{j=1 \\ \ell=j+1}} 
	z_j z_{\ell},\, -\sum\limits_{j=1}^5 z_j, \, 1 \right]^T}{\prod\limits^5_{j=1}  (z_k - z_j)},
\end{equation}
where all of the above sums and products exclude the index $k$.
Therefore, the eigenstructure hinges on the values of the five numbers: $z_k$ for $k=1,2,3,4,5$.
Unfortunately, we are not able obtain these quantities in closed form.
However, it is possible to calculate approximate values for these quantities in certain
asymptotic limits. In particular, one important limit that is useful later on in this work
is the near thermodynamic equilibrium limit, which is characterized here by $\mu_1 \approx \mu_2$.
In this limit the five distinct $z_k$'s are given by
\begin{align}
   z_1, z_2 &= u + \frac{2 \tilde{q}}{5p} - 
 \sqrt{\frac{\left(5 \pm \sqrt{10} \right) p (1 - \alpha)}{\rho}}
 + {\mathcal O}\left( | \mu_2 - \mu_1 |^2 \right), \\
 z_3 &= u + \frac{2 \tilde{q}}{5p} + {\mathcal O}\left( | \mu_2 - \mu_1 |^3 \right), \\
 z_4, z_5 &= u + \frac{2 \tilde{q}}{5p} + 
 \sqrt{\frac{\left(5 \mp \sqrt{10} \right) p (1 - \alpha)}{\rho}}
 + {\mathcal O}\left( | \mu_2 - \mu_1 |^2 \right).
\end{align}
Assuming that $\rho>0$ and $p>0$ and noting that $\alpha \approx 0$ if $\mu_1 \approx \mu_2$,
we see that in this limit the system is strongly hyperbolic. 
We can also approximately compute 
\begin{align}
\nabla_U \lambda^{(k)} \cdot r^{(k)} = 
\frac{c_k q p^2  (1 - \alpha)^2 (14 p^3 \alpha^3 + 5 \rho q^2)}{20 \rho (4 p^3 \alpha^3 
+\rho q^2)^2} + {\mathcal O}\left( | \mu_2 - \mu_1 |^{-2} \right),
\end{align}
where 
\[
c_k = \left\{ 2+\sqrt{10}, \, 2-\sqrt{10},  \, \frac{3}{4}, \, 2-\sqrt{10}, \, 2+\sqrt{10} \right\}.
\]
We note that in this limit $\nabla_U \lambda^{(k)} \cdot r^{(k)}$  changes sign
if $q$ changes sign, which means that there must exist some state, $U^{\star}$,
for which $\nabla_U \lambda^{(k)} \cdot r^{(k)} = 0$. This shows that the waves
in this system are {\it non-genuinely nonlinear} and may admit composite
wave solutions (e.g., see \cite{article:Liu74}).

In order to illustrate how the bi-Gaussian behaves, we show numerical
solutions of two Riemann problems using the discontinuous Galerkin scheme as described in
\S\ref{sec:dg-method} -\S\ref{sec:riemann}  . In Figure \ref{fig:riemann} we show the
solution to a Riemann problem with a heat flux, $q$, that is striclty positive. 
In this problem the solution is a set of five classical waves (i.e., shocks and rarefactions).
In contrast, in Figure \ref{fig:chalons} we show the solution to a Riemann problem
with heat flux, $q$, that changes sign over the simulation domain.  
In this case we see a solution
with two compound waves, each of which is a rarefaction connected to a shock,
 propagating in opposite directions.

\begin{figure}
\begin{tabular}{cc}
   (a)\includegraphics[width=51mm]{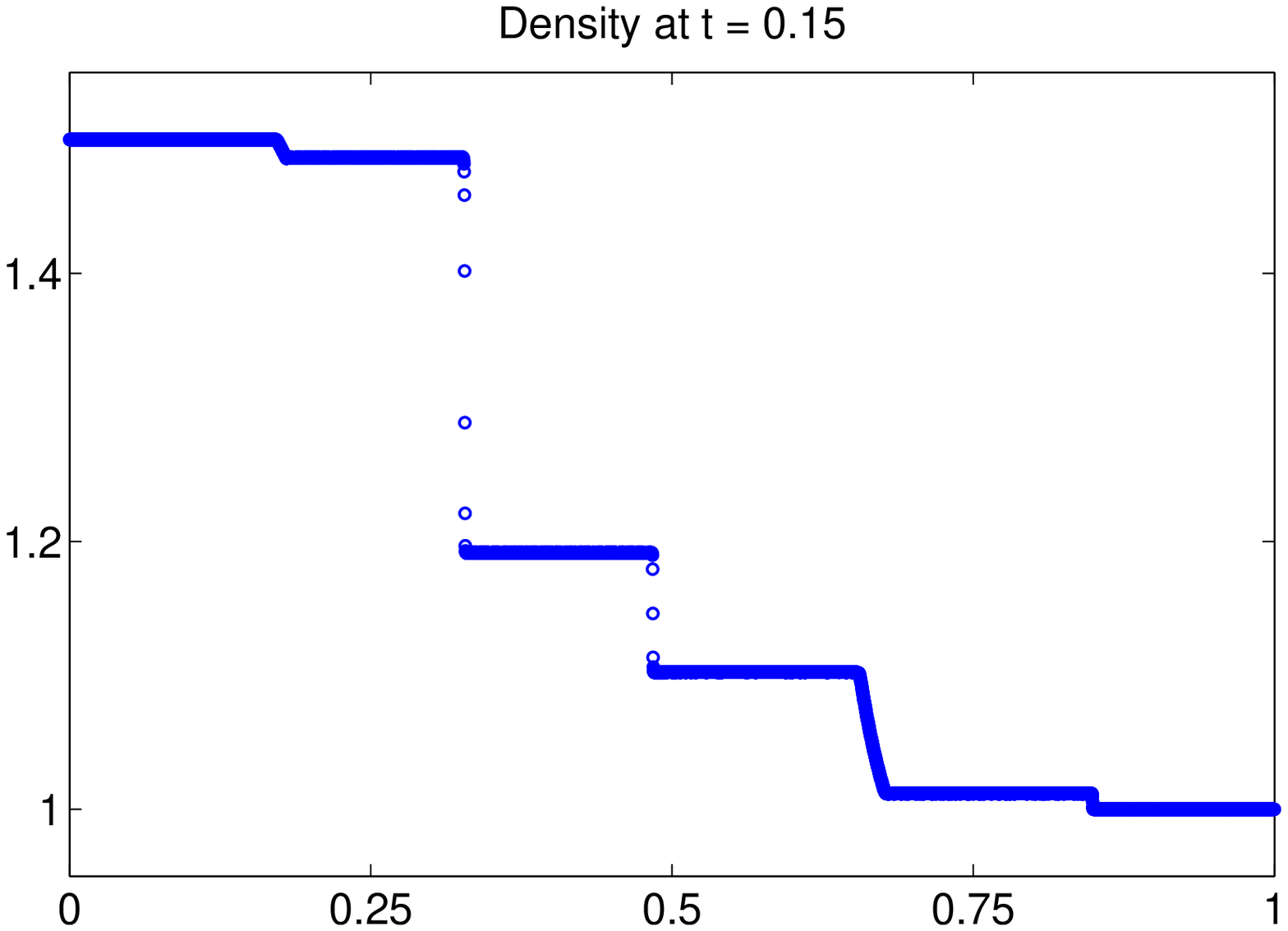} &
   (b)\includegraphics[width=51mm]{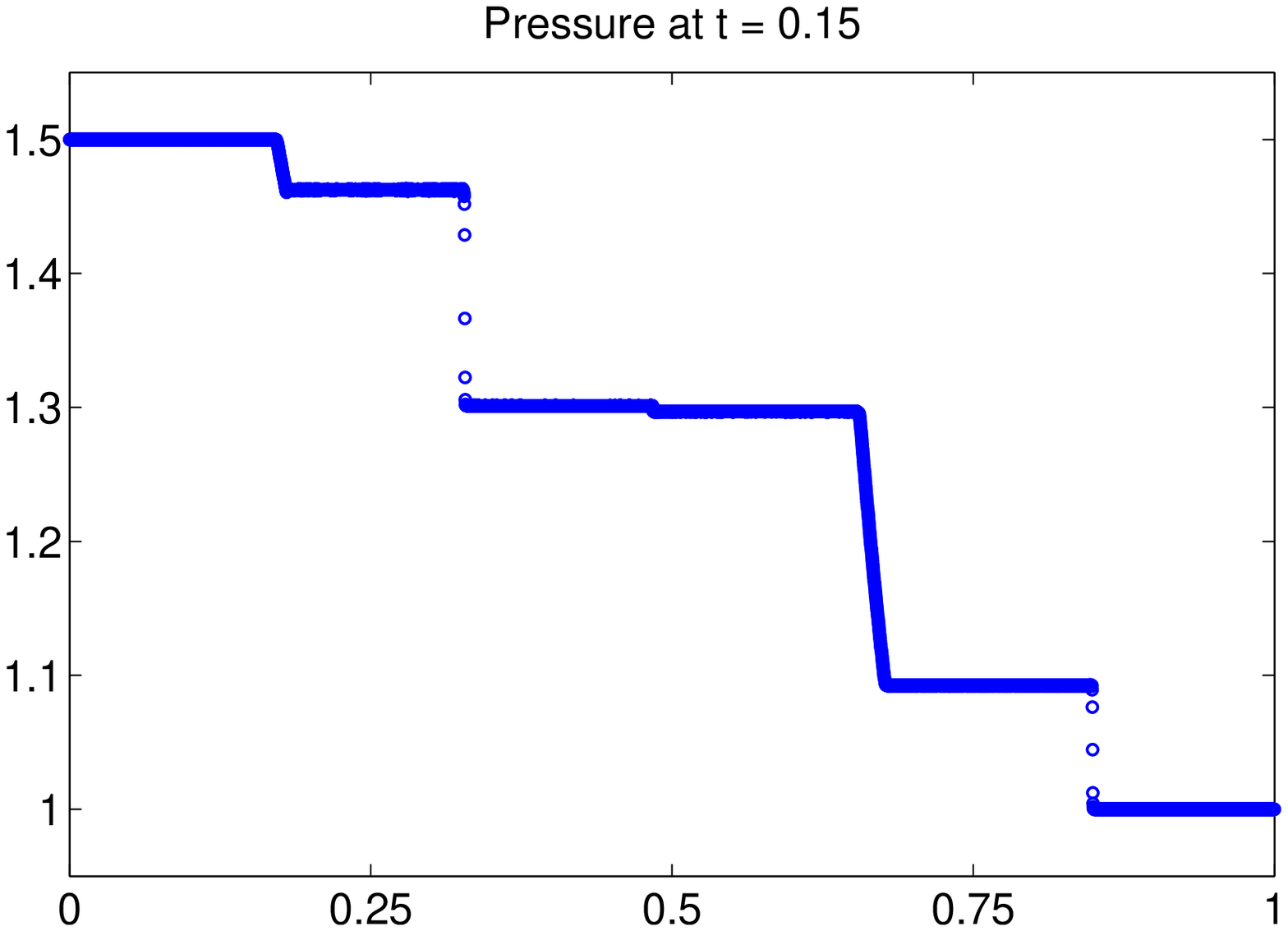} \\
   (c)\includegraphics[width=51mm]{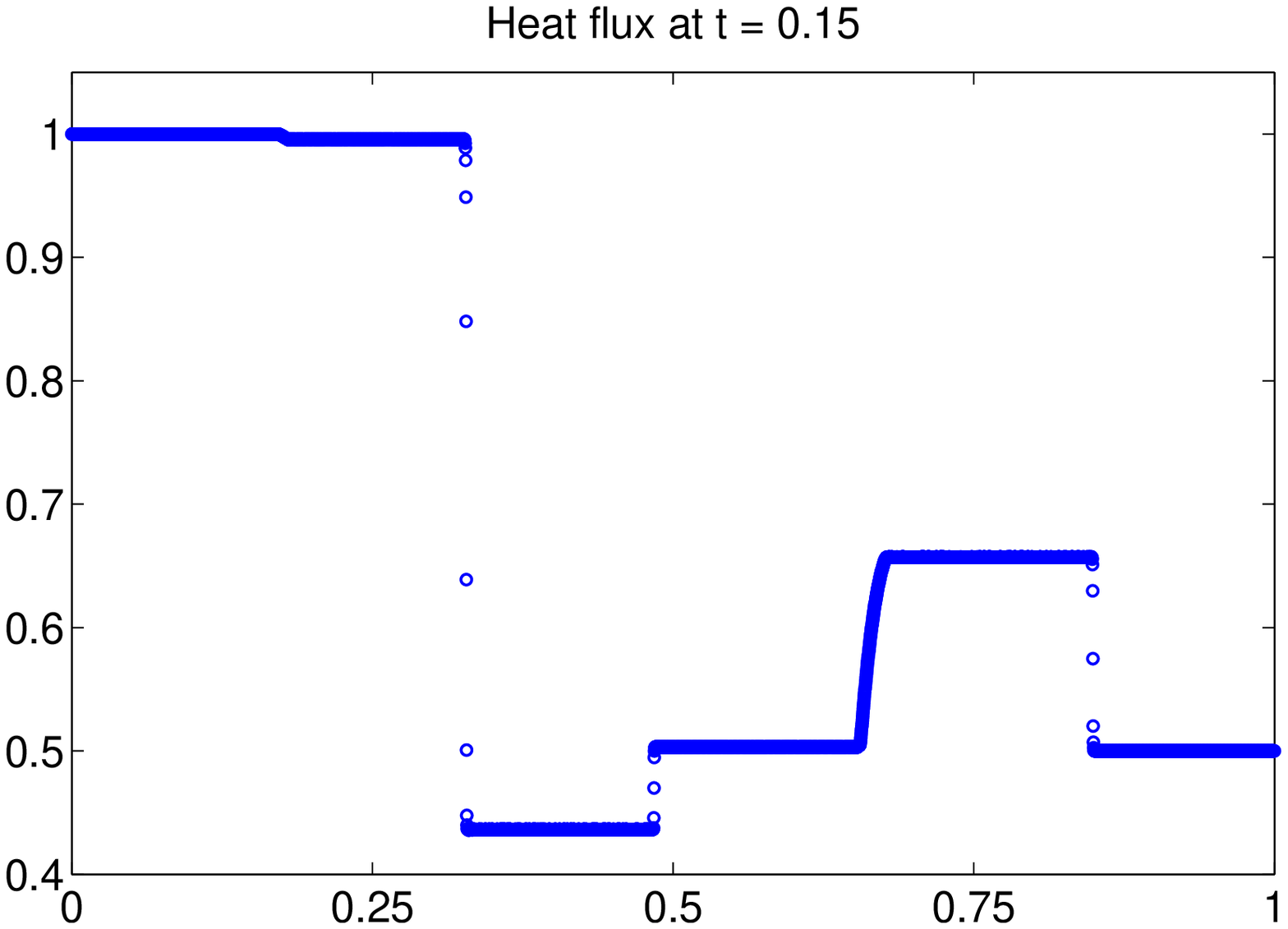} &
    (d)\includegraphics[width=51mm]{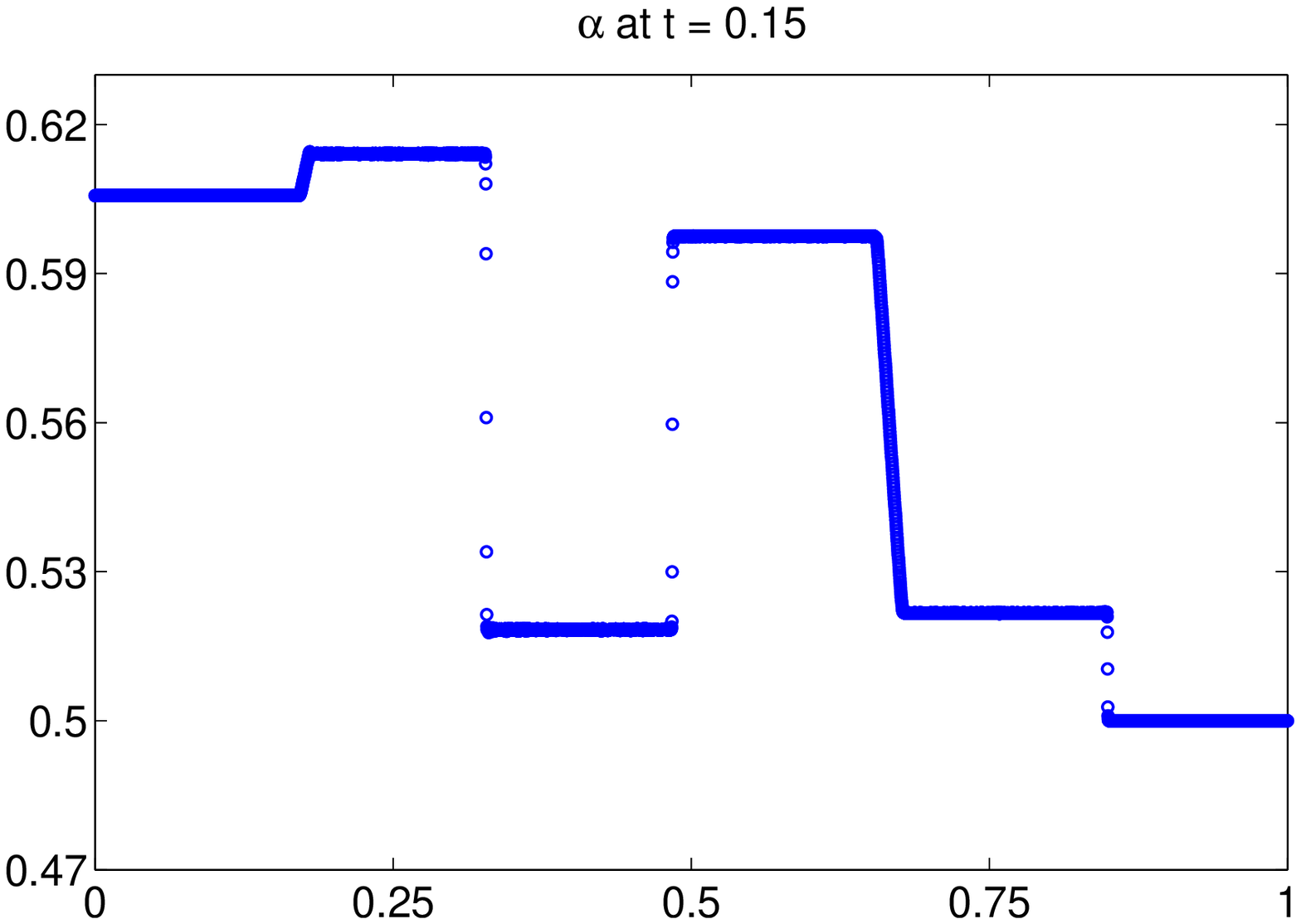} \\
   (e)\includegraphics[width=51mm]{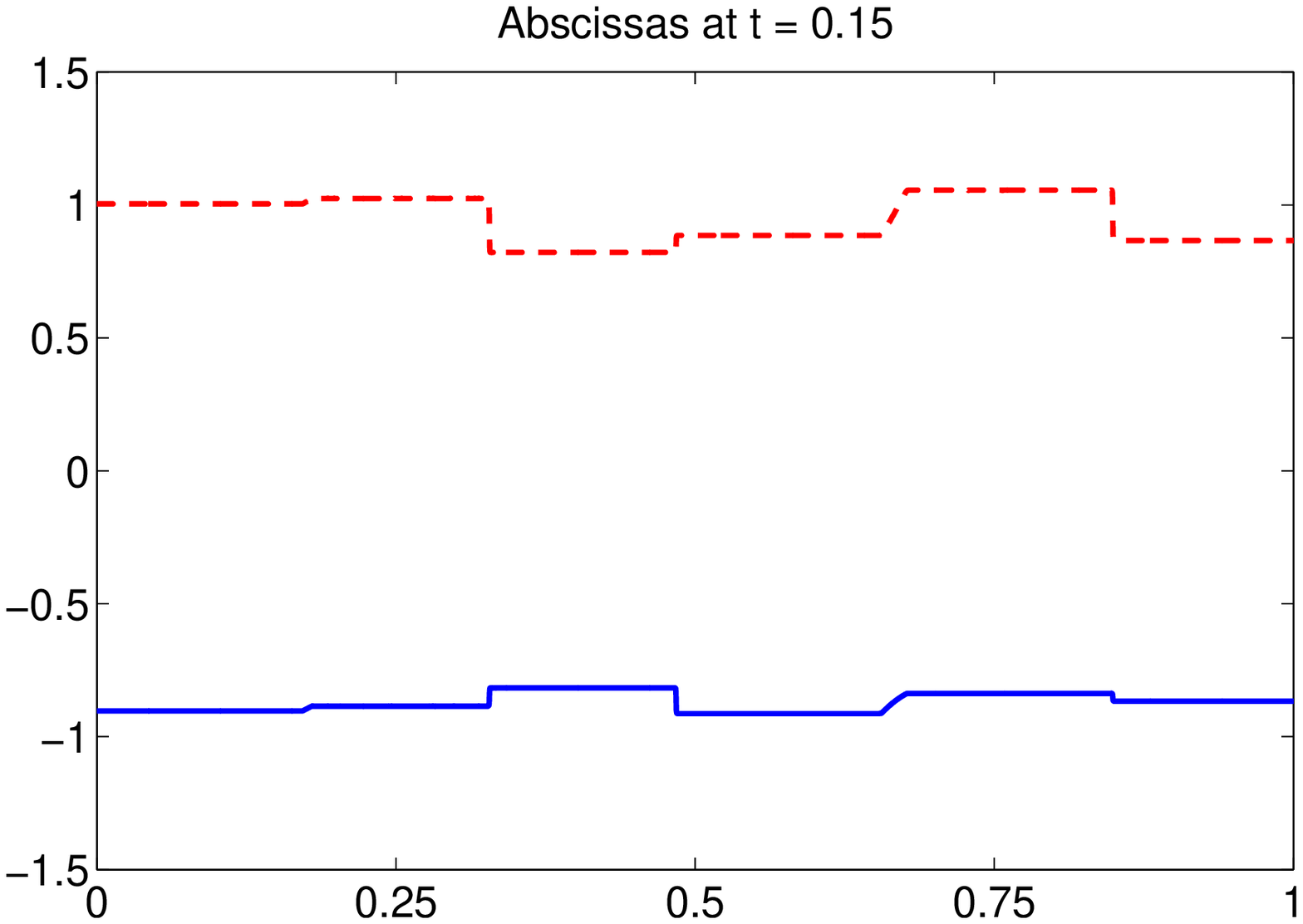} &
   (f)\includegraphics[width=51mm]{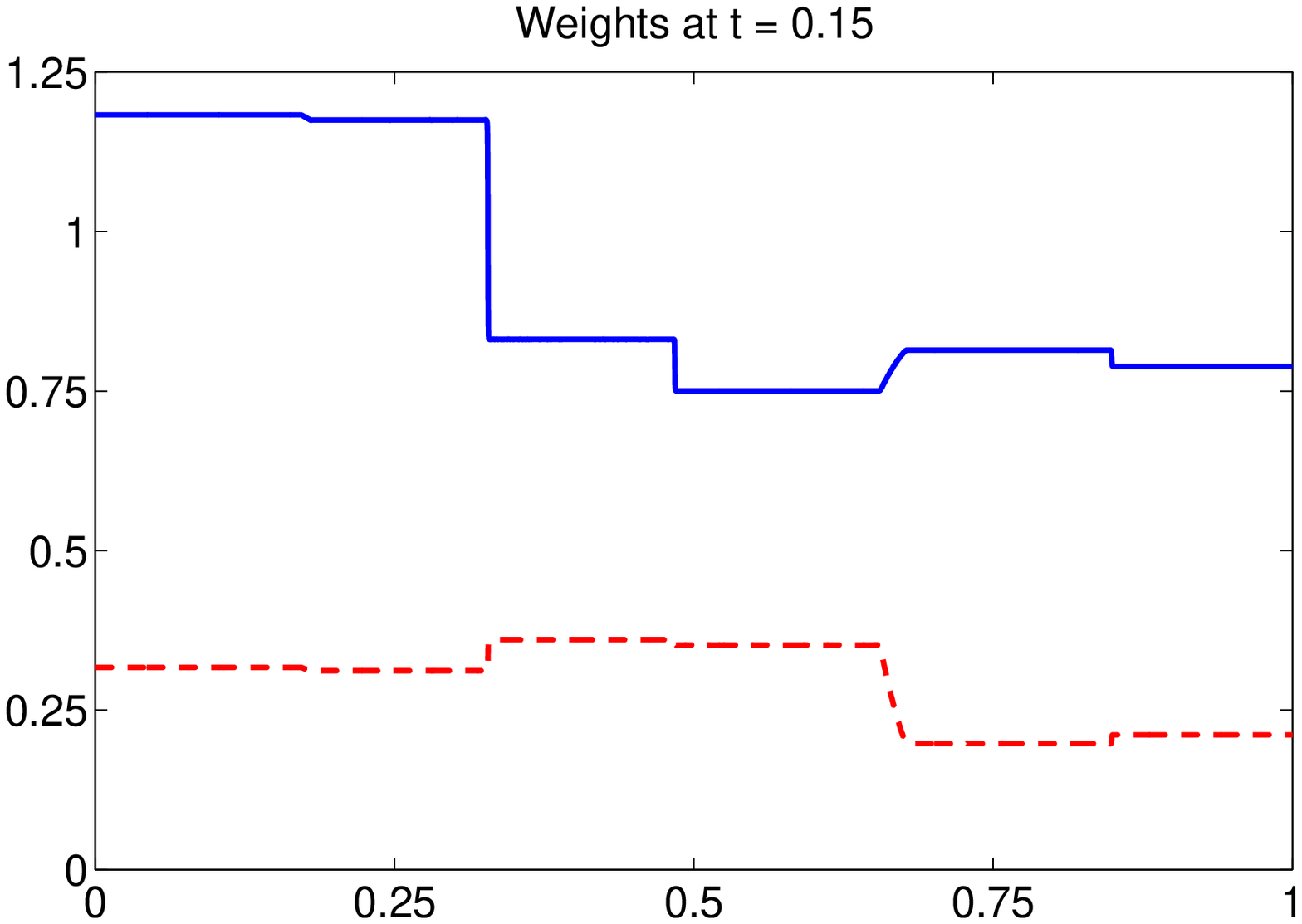}
\end{tabular}
  \caption{A shock tube problem for the bi-Gaussian system. In this example
  the initial states are $(\rho, u, p, q, r)_{\text{left}} = (1.5, -0.5, 1.5, 1.0, 4.5)$ and
  $(\rho, u, p, q, r)_{\text{right}} = (1.0, -0.5, 1.0, 0.5, 3.0)$. This data is chosen so
  that $q>0 \, \forall x,t$, ensuring that we do not encounter points where the
  convexity changes. Shown in these panels are the (a) density ($\rho$), (b) pressure ($p$),
  (c) heat flux ($q$), (d) width parameter ($\alpha$), (e) quadrature abscissas ($\mu_1, \, \mu_2$), and
  (f)  quadrature weights ($\omega_1, \, \omega_2$). The resulting solution shows,
  counting waves from left to right, a 1-rarefaction, 2-shock, 3-shock, 4-rarefaction, and 5-shock.
  \label{fig:riemann}}
\end{figure}

\begin{figure}
\begin{tabular}{cc}
   (a)\includegraphics[width=51mm]{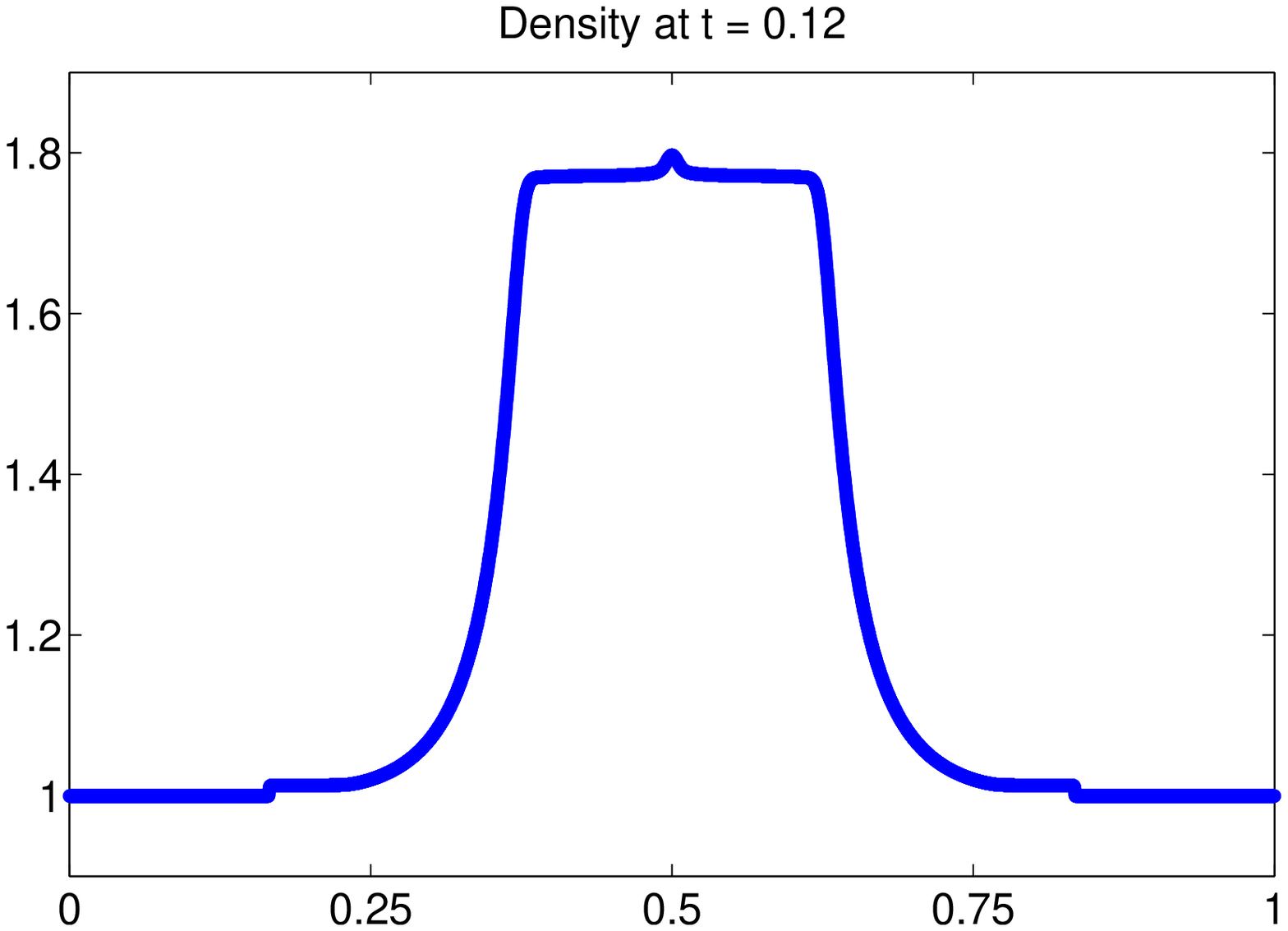} &
   (b)\includegraphics[width=51mm]{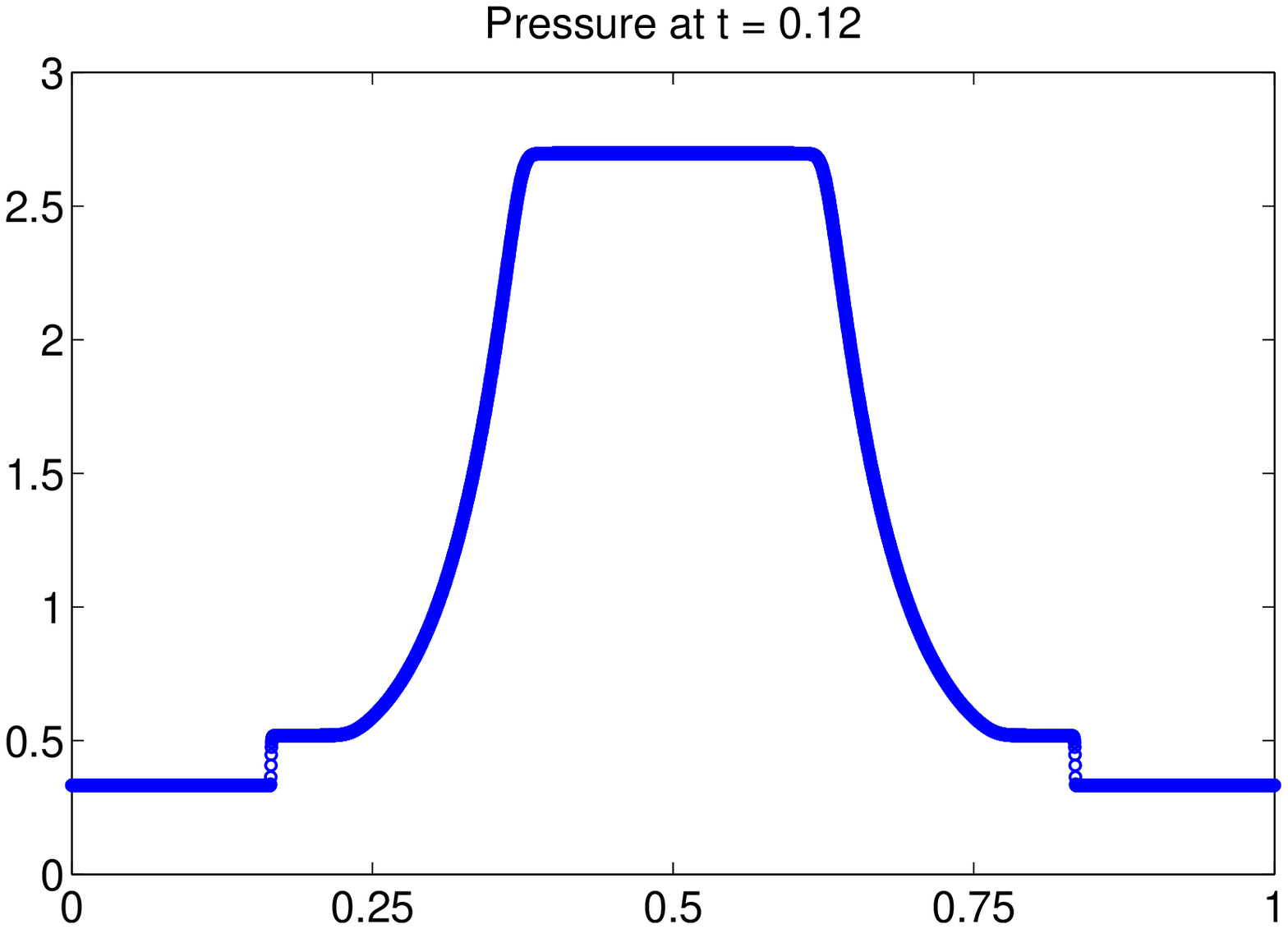} \\
   (c)\includegraphics[width=51mm]{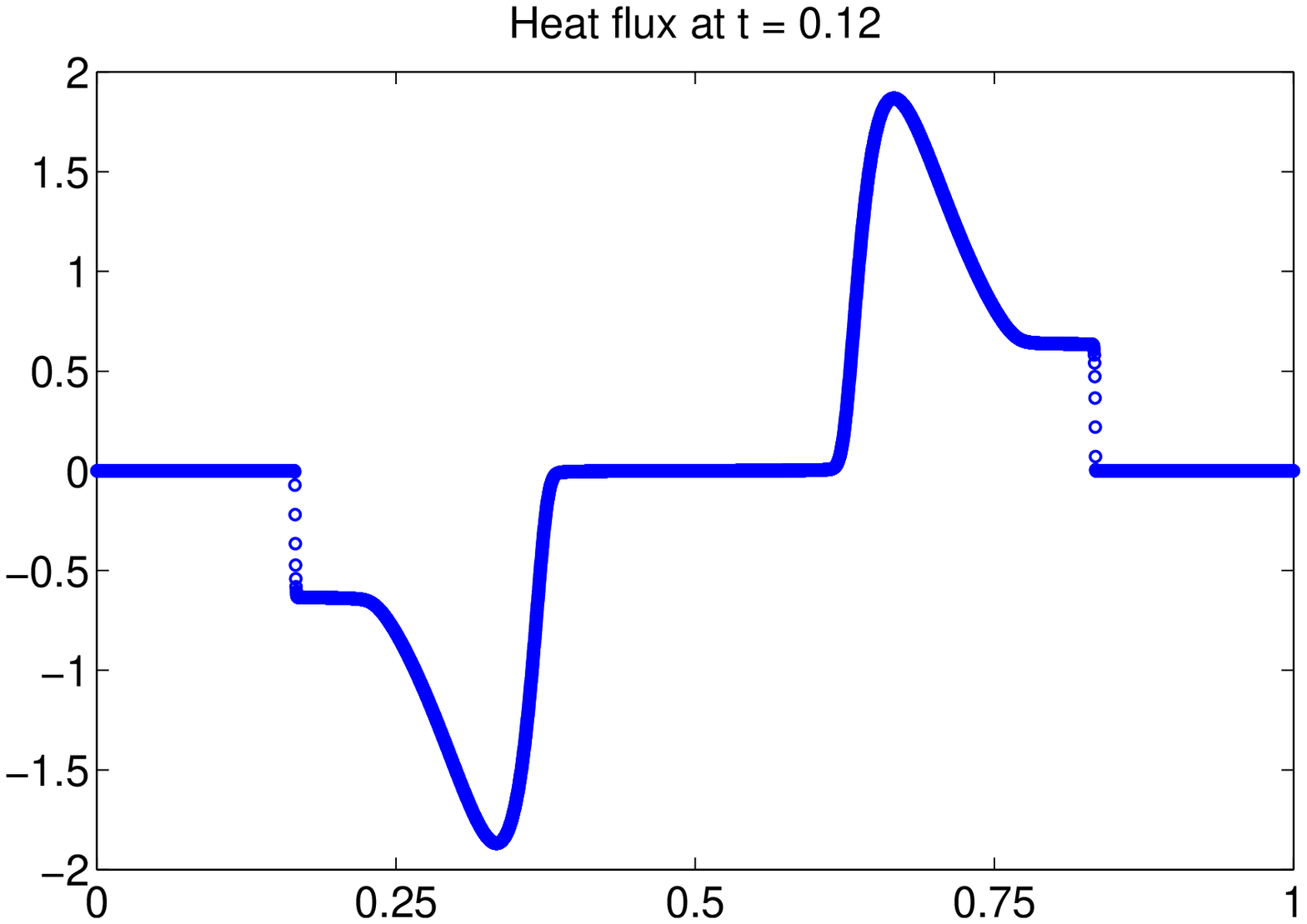} &
    (d)\includegraphics[width=51mm]{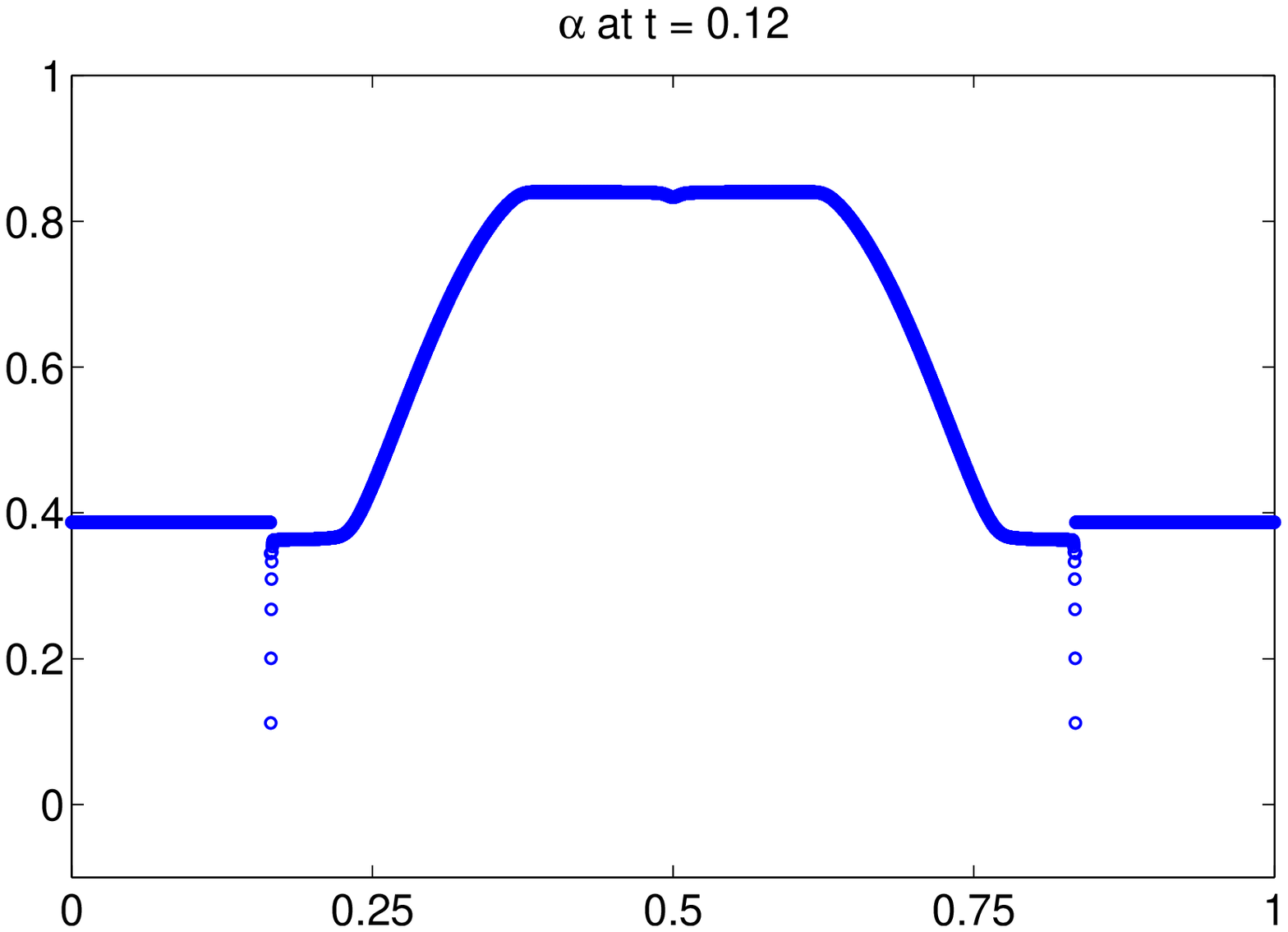} \\
   (e)\includegraphics[width=51mm]{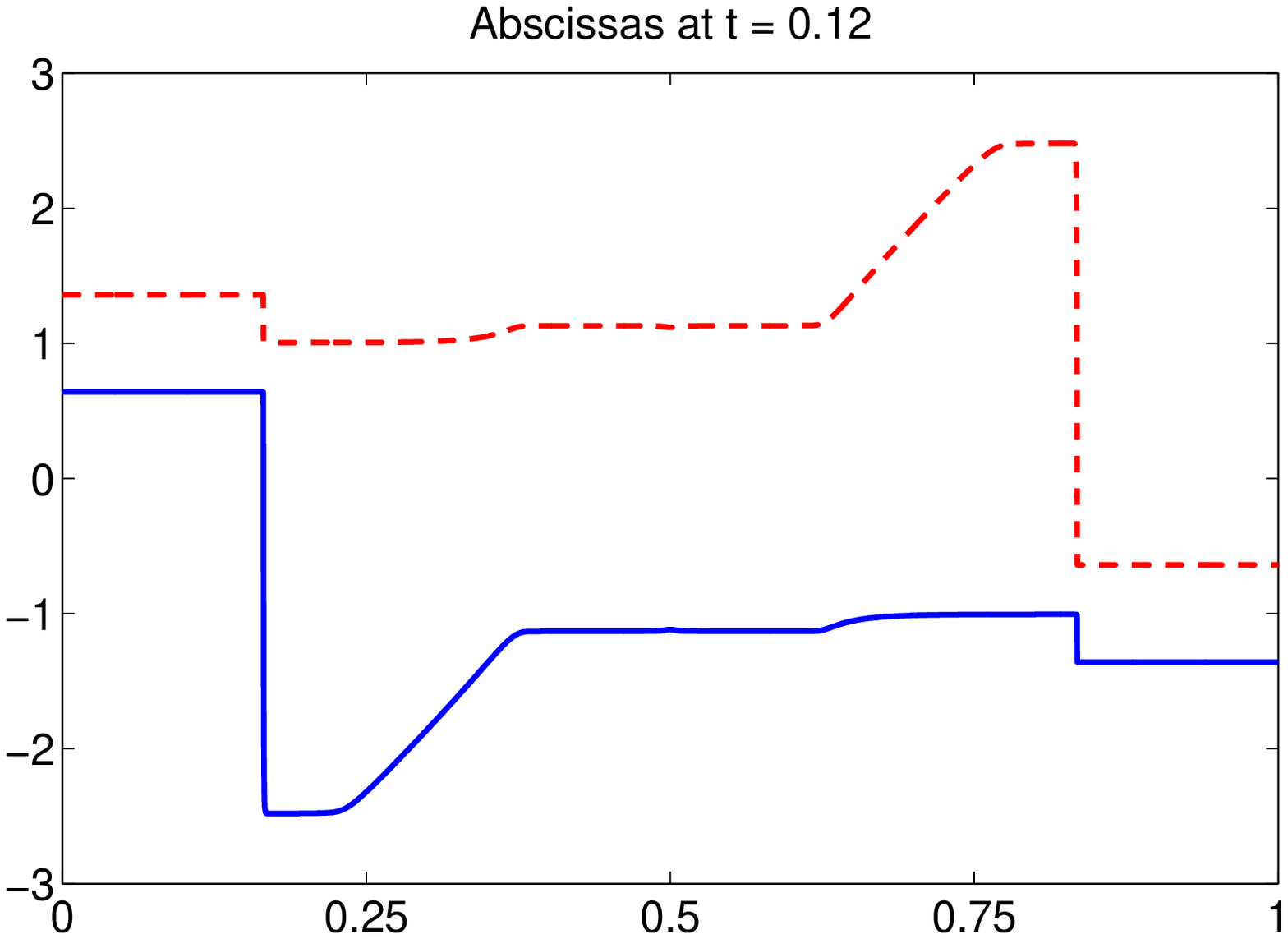} &
   (f)\includegraphics[width=51mm]{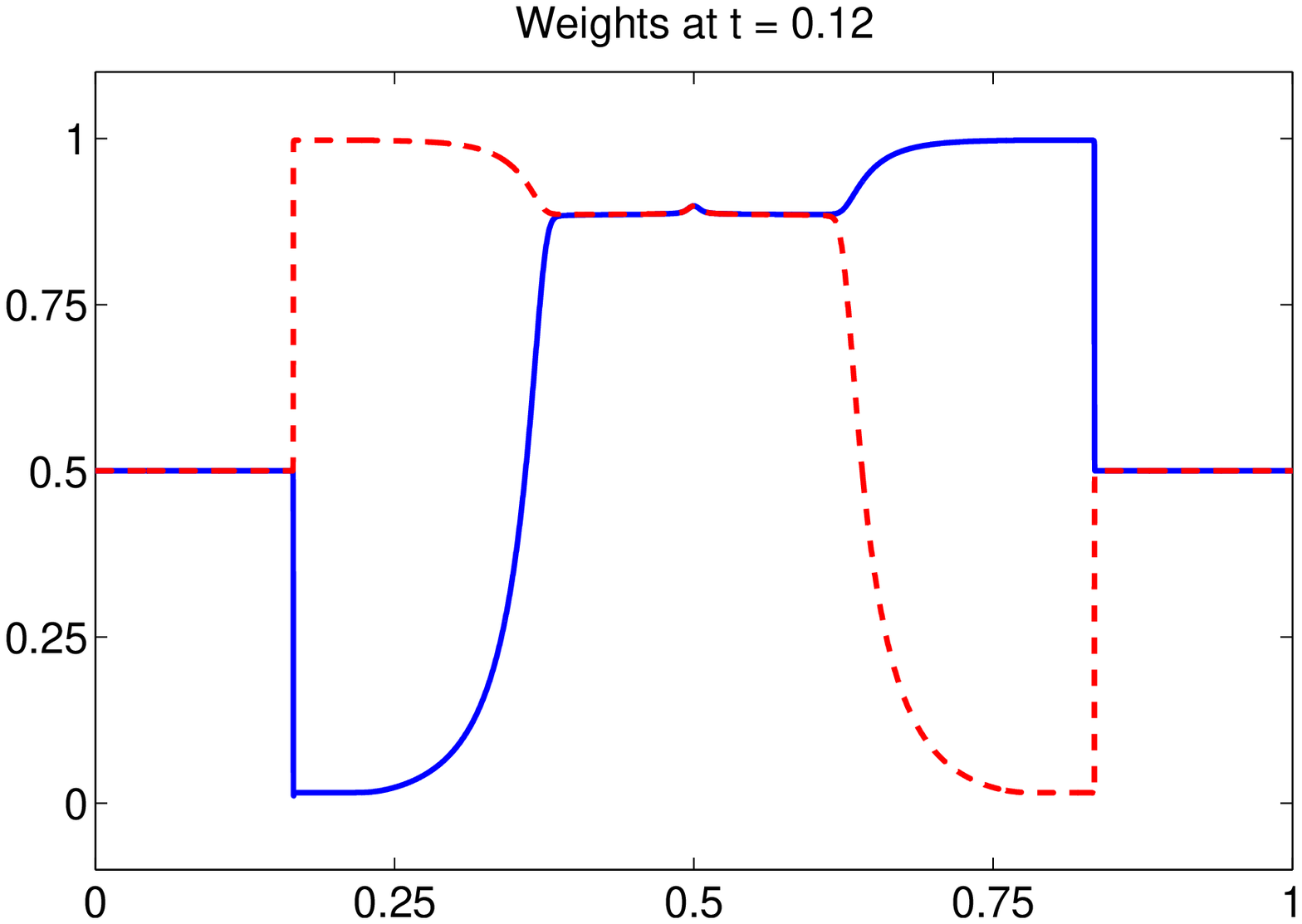}
\end{tabular}
  \caption{A shock tube problem for the bi-Gaussian system. In this example
  the initial states are $(\rho, u, p, q, r)_{\text{left}} = (1, 1, 1/3, 0, 0.3)$ and
  $(\rho, u, p, q, r)_{\text{right}} = (1, -1, 1/3, 0, 0.3)$. This problem is similar to the one
  shown in Chalons et al. \cite{article:Chalons10}; however, we have  reduced the initial value of $r$ from $1/3$ to $0.3$ in 
  order to emphasis the 1-shock and 5-shock. Shown in these panels are the (a) density ($\rho$), (b) pressure ($p$),
  (c) heat flux ($q$), (d) width parameter ($\alpha$), (e) quadrature abscissas ($\mu_1, \, \mu_2$), and
  (f)  quadrature weights ($\omega_1, \, \omega_2$). The resulting solution shows
  a 1-shock connected to a 2-rarefaction, as well as a 4-rarefaction connected to a 5-shock. We note that having
  compound waves is typical of systems with non-convex fluxes (see for example pages 350--357 of LeVeque
  \cite{book:Le02}.
  \label{fig:chalons}}
\end{figure}

\subsection{Quadrature using ${\mathcal C}^{0}$ B-splines}
\label{sec:hat}
From the above discussion we can view the bi-delta and the bi-Gaussian closure
methods as members of the same family of methods, where the bi-delta is one
extreme (compactly supported distributions with zero width) and the bi-Gaussian is the other (non-compactly supported
distributions with maximal standard deviation as allowed in this representation). 
Using this point-of-view, we can also construct methods that are in-between these
two extremes. One simple example of this is a bi-distribution representation based on
${\mathcal C}^{0}$ B-splines:
\begin{align}
\label{eqn:bsplines}
  \overline{f}(t,x,v) = \omega_1 B^0_{\sigma} \left( v - \mu_1 \right) 
  	+ \omega_2 B^0_{{\sigma}} \left( v - \mu_2 \right),
\end{align}
where
\begin{align}
  B^0_{{\sigma}}(v) = \begin{cases}
  	\frac{2}{\sigma} \left( 2 v + {\sqrt{\sigma}} \right) & \text{if} \, \, -{\sqrt{\sigma}} \le 2v \le 0, \\
	\frac{2}{\sigma} \left( \sqrt{\sigma} - 2v  \right) & \text{if} \, \,  0 \le 2v \le {\sqrt{\sigma}}, \\
	0 & \text{otherwise}.
  \end{cases}
\end{align}
If we introduce the parameter 
\begin{equation}
   \sigma = \frac{24p}{\rho} \left( 1 - \alpha \right),
\end{equation}
and force \eqref{eqn:bsplines} to have as its first five moments $M_0$,
 $M_1$, $M_2$, $M_3$, $M_4$, then we again arrive at system
 \eqref{eqn:MG1}--\eqref{eqn:MG4}, but now with a slightly modified
 equation to match the $M_4$ moment:
 \begin{equation}
 \label{eqn:MG5_bspline}
 \omega_1 \mu_1^4 + \omega_2 \mu_2^4 = 
 	\rho u^4 + 4 q u + 6 \alpha p u^2  + r +  \frac{6p^2}{5\rho} (3 \alpha+2) (\alpha-1).
 \end{equation}
 \begin{theorem}[Moment-realizability condition]
\label{thm:mom-realz-bspline}
Assume that the primitive variables satisfy the following conditions:
\begin{itemize}
\item Positive density: \quad $0< \rho$,
\item Positive pressure: \quad $0<p$,
\item Bound on $r$: \quad $\frac{q^2}{p} + \frac{p^2}{\rho} \le r \le \frac{13 q^2}{3 p} + \frac{33 p^2}{13 \rho}$.
\end{itemize}
Then there exists a unique 
$\alpha \in \left[\frac{3}{13},1\right]$ that satisfies the following cubic polynomial:
\begin{equation}
\label{eqn:poly-alpha-bspline}
{\mathcal P}(\alpha) = 13 p^3 \alpha^3 - 6 p^3 \alpha^2 + \alpha p \left( 5 r \rho-12 p^2 \right)
- 5 \rho q^2 = 0.
\end{equation}
Using this $\alpha$ and the definitions for the abscissas and weights
given by \eqref{eqn:mus} and \eqref{eqn:omegas}, we can 
exactly solve the moment inversion problem given by equations \eqref{eqn:MG1}--\eqref{eqn:MG4} and
\eqref{eqn:MG5_bspline}.
\end{theorem}

\begin{proof}
The abscissas and weights
given by \eqref{eqn:mus} and \eqref{eqn:omegas} automatically satisfy  \eqref{eqn:MG1}--\eqref{eqn:MG4}
for any $\alpha \in (0,1]$. Therefore, the only thing left to do is to satisfy equation \eqref{eqn:MG5_bspline}.
Using \eqref{eqn:mus} and \eqref{eqn:omegas}, one finds that \eqref{eqn:MG5_bspline} reduces
to the cubic polynomial  \eqref{eqn:poly-alpha-bspline}. We compute the following:
\begin{gather*}
	{\mathcal P}\left( \frac{3}{13} \right) = -5 \rho q^2- \frac{495 p^3}{169}  + \frac{15 r \rho p}{13}  \le 0,
	\quad
	{\mathcal P}\left( 1 \right) = -5 \rho q^2 - 5 p^3 + 5 r \rho p \ge 0, \\
	\text{and} \quad {\mathcal P}''(\alpha) > 0 \quad \forall \alpha> \frac{2}{13},
\end{gather*}
which completes the proof.
\end{proof}

\begin{remark}
Because the bi-B-spline ansatz is compactly supported, there is a maximum value of $r$ that
can represented. In particular, if $M_0$ through $M_4$ represent the moments of
a Gaussian distribution, the value of $r=3p^2/\rho$ will exceed the maximum allowed value
in the above theorem. In practice we can remedy this situation by taking
$\alpha = \frac{3}{13}$ (i.e., the minimum allowed value) whenever $r$ exceeds
the maximum allowed value of $\frac{13 q^2}{3 p} + \frac{33 p^2}{13 \rho}$.
\end{remark}

The moment closure imposed by the bi-B-spline ansatz is the
following replacement for the $M_5$ moment:
\begin{equation}
\begin{split}
\overline{M}_5 &= \rho u^5 + 10 u^2 \left( p u +  q \right)
 + \frac{2 p q}{\rho} \left( 5 - 4 \alpha \right) \\
 &+  \frac{ p^2 u}{\rho} \left(12 +  6\alpha -13\alpha^2 \right)
 +  \frac{5 q^2 u}{p \alpha} 
 + \frac{q^3}{p^2 \alpha^2}.
 \end{split}
\end{equation}
The bi-B-spline moment-closure has two advantages over the
bi-Gaussian system: (1) $\alpha$ is uniformly bounded away
from 0, which means we don't have to worry about the same
degeneracies as in the bi-Gaussian case; and (2)
the distribution function $\overline{f}(t,x,v)$ is compactly supported
and piecewise linear, which makes computing integrals such
as those needed in the flux-vector splitting method
described in \S \ref{sec:riemann} simpler.
In Figure \ref{fig:riemann_bspline} we show a simulation using the bi-B-spline moment-closure
on the same initial data as used in Figure \ref{fig:riemann}.
These results show, at least in the case when $\alpha>1/3$,
that the bi-Gaussian and bi-B-spline moment-closure approaches
produce similar Riemann solutions.

\begin{figure}
\begin{tabular}{cc}
   (a)\includegraphics[width=51mm]{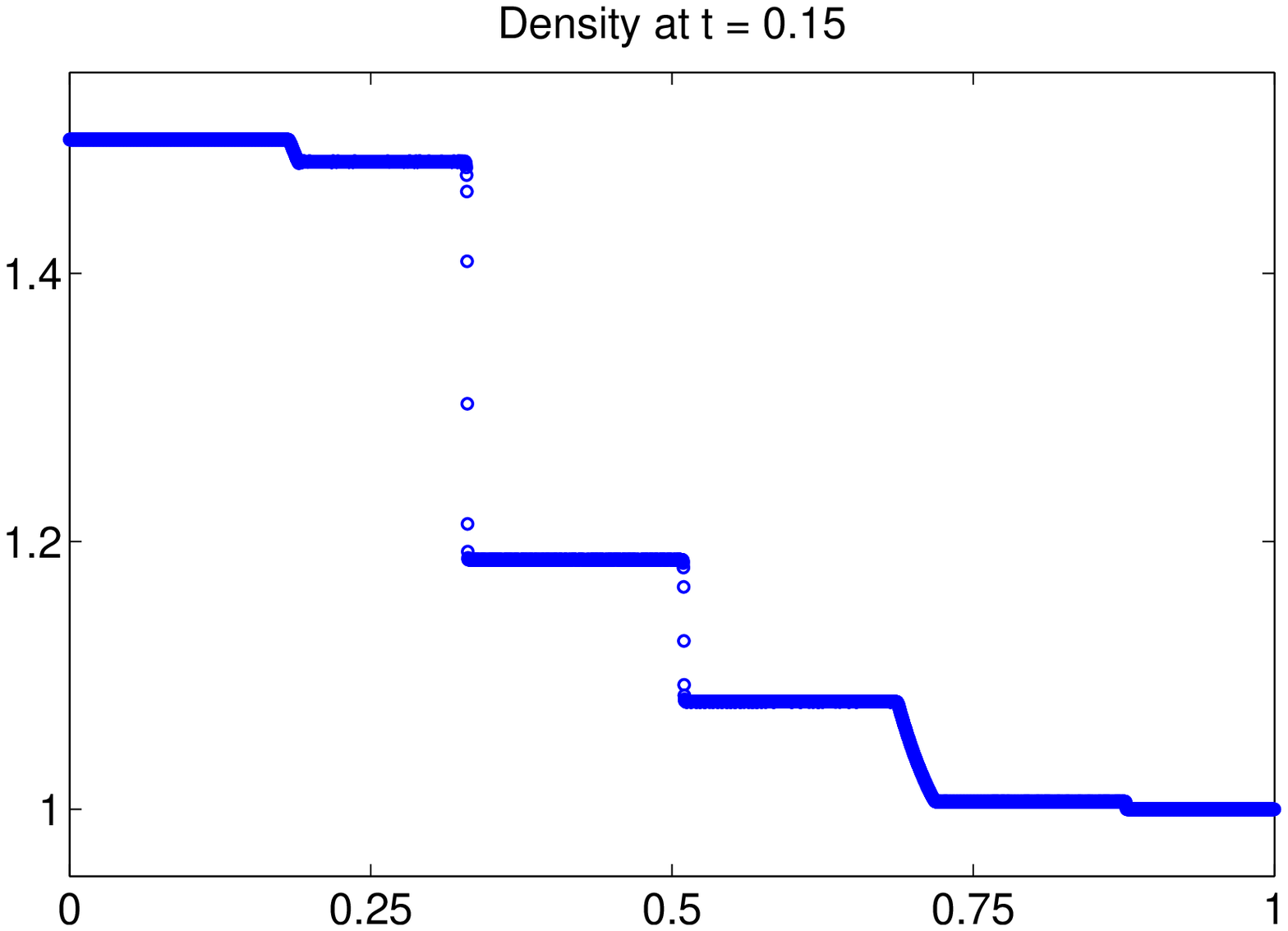} &
   (b)\includegraphics[width=51mm]{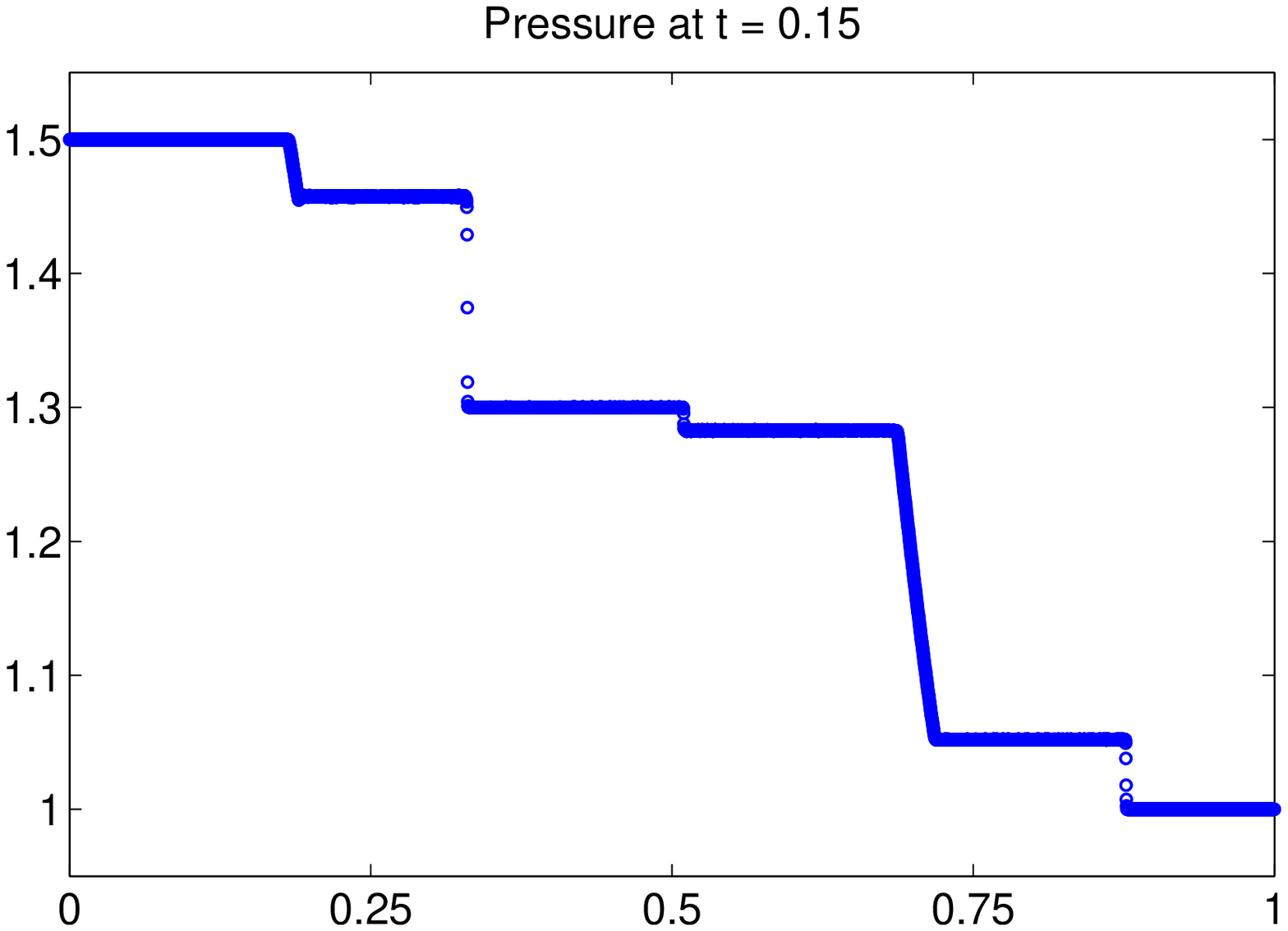} \\
   (c)\includegraphics[width=51mm]{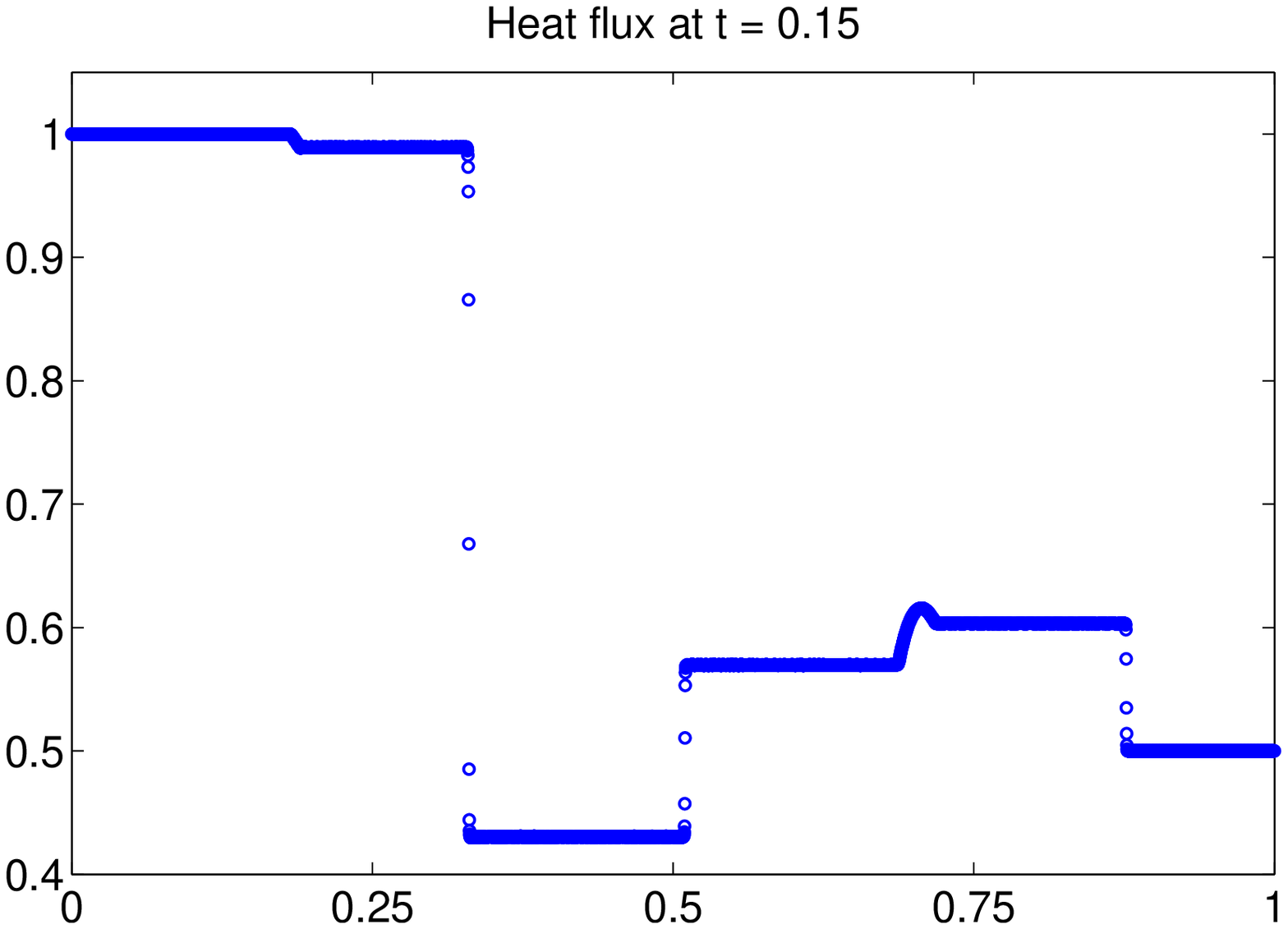}  &
    (d)\includegraphics[width=51mm]{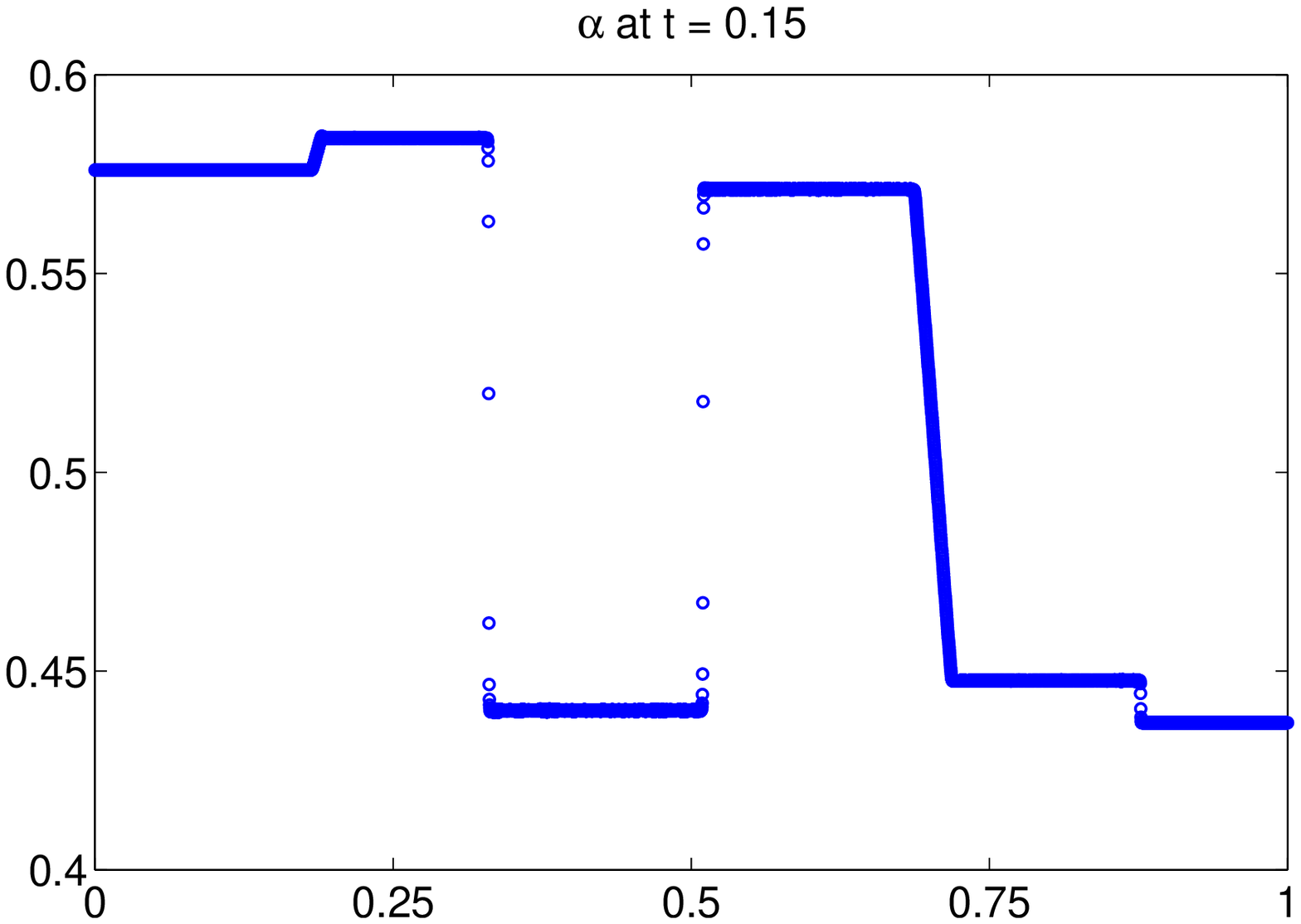} \\
   (e)\includegraphics[width=51mm]{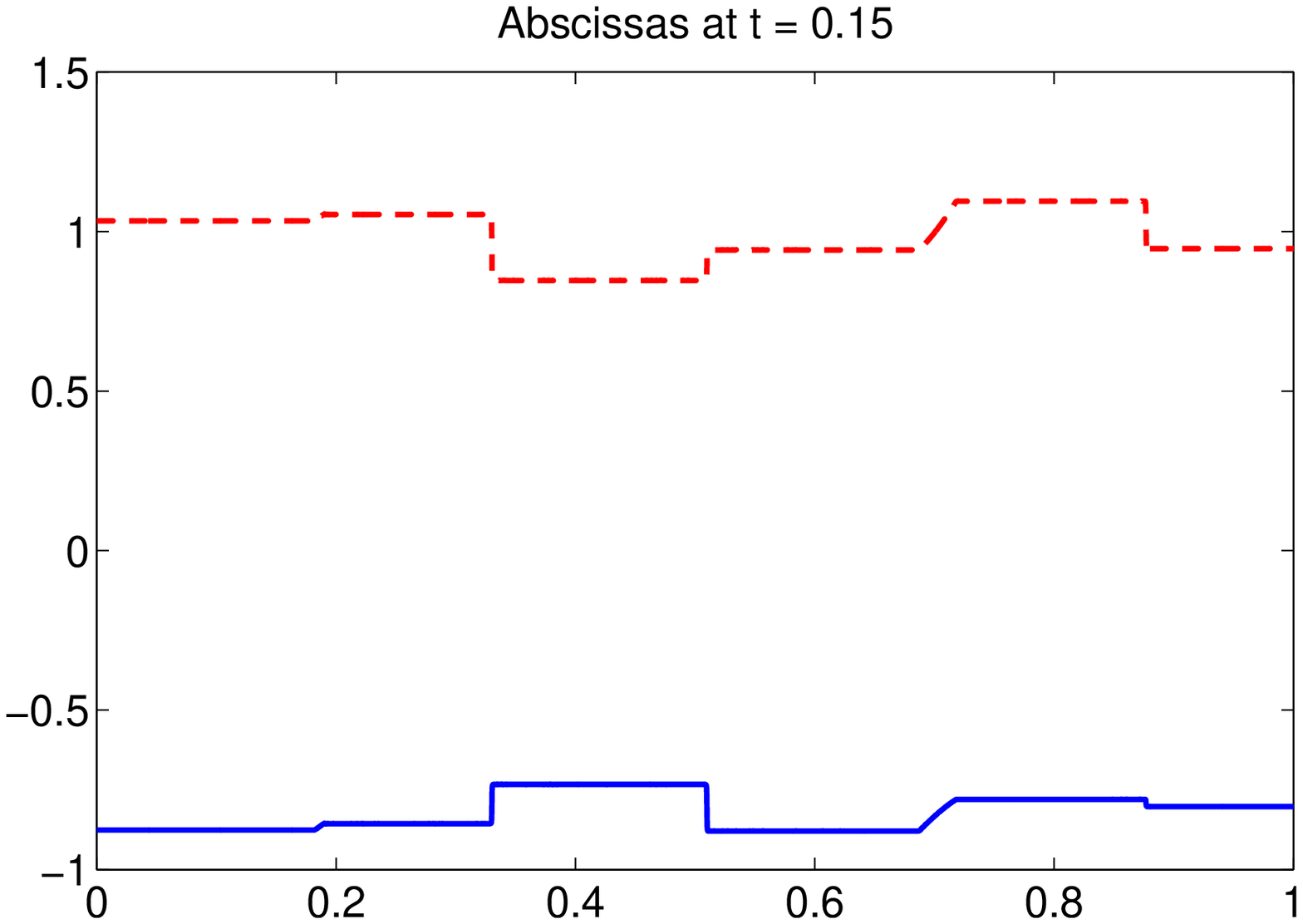} &
   (f)\includegraphics[width=51mm]{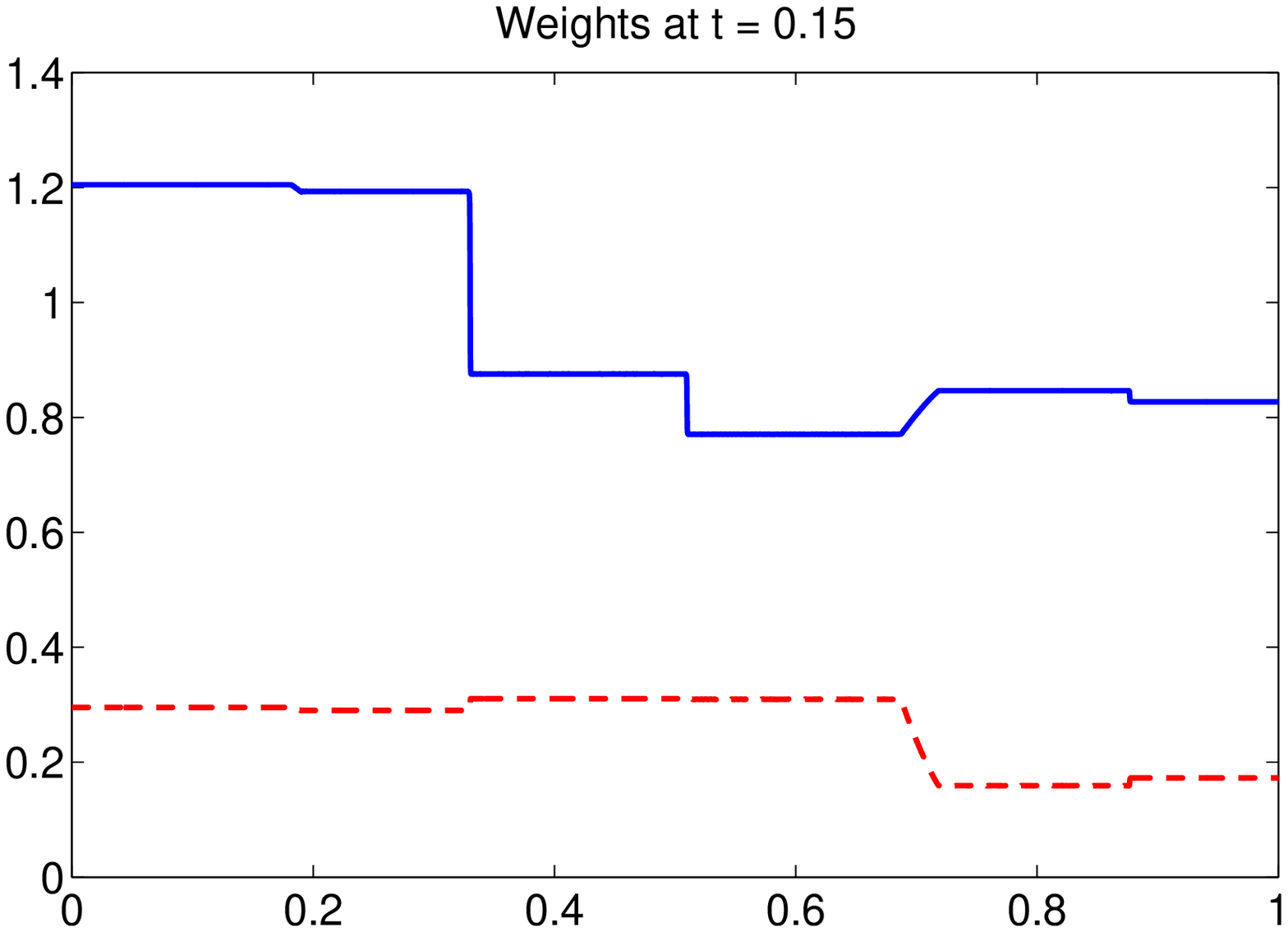}
\end{tabular}
  \caption{A shock tube problem for the bi-B-spline system. In this example
  the initial states are $(\rho, u, p, q, r)_{\text{left}} = (1.5, -0.5, 1.5, 1.0, 4.5)$ and
  $(\rho, u, p, q, r)_{\text{right}} = (1.0, -0.5, 1.0, 0.5, 3.0)$. 
  Shown in these panels are the (a) density ($\rho$), (b) pressure ($p$),
  (c) heat flux ($q$), (d) width parameter ($\alpha$), (e) quadrature abscissas ($\mu_1, \, \mu_2$), and
  (f)  quadrature weights ($\omega_1, \, \omega_2$). The resulting solution shows,
  counting waves from left to right, a 1-rarefaction, 2-shock, 3-shock, 4-rarefaction, and 5-shock.
  \label{fig:riemann_bspline}}
\end{figure}

%\subsection{Extension to higher moments}
%\label{sec:higher}
%Dude.

\section{DG quadrature-based moment-closure schemes for VPFP}
\label{sec:numerical-method}
We describe in this section an application of the quadrature-based
moment approach as described in previous sections to a particular
set of equations from plasma physics: the Vlasov-Poisson-Fokker-Planck (VPFP)
equations \eqref{eqn:VPFP}. Wang and Jin \cite{article:WaJi11} recently
developed an {\it asymptotic-preserving} scheme for the VPFP equation, 
where they modified a fully kinetic solver for
VPFP so that it remains asymptotic preserving in the high-field limit
$\varepsilon \rightarrow 0^+$. Although the Wang and Jin approach
has the nice property that it can be applied for {\it any} value of
$\varepsilon>0$, if one is really interested in regimes where $\varepsilon$
is relatively small (i.e., {\it near} thermodynamic equilibrium), then their
approach is computationally expensive (i.e., requires solving a PDE in
2D rather than 1D). Our focus in this section is on approximately
solving the VPFP system using quadrature-based moment-closure techniques
that remain asymptotic-preserving in the high-field limit
$\varepsilon \rightarrow 0^+$. This approach allows us to efficiently capture
{\it near} thermodynamic equilibrium solutions.

\subsection{Strang operator splitting}
\label{sec:strang}
Wang and Jin \cite{article:WaJi11} achieve an asymptotic-preserving
scheme through the use of a clever semi-implicit time discretization.
In this work we make use of a more standard trick: Strang operator splitting \cite{article:St68},
which has been used for Vlasov-Poisson simulations since the work of
Cheng and Knorr \cite{article:ChKn76}. In particular, Schaeffer \cite{article:Sch98}
modified the Cheng and Knorr approach to construct an efficient
method for VPFP.

In our approach we use a Strang splitting for the VPFP system under
a quadrature-based moment-closure with the following 
steps:  
\algsetup{indent=2em}
\begin{algorithmic}
\STATE 1. Solve the Poisson equation:
 \[
 	-\phi_{,x,x} = \rho_{0}(x) - \rho^{n}(x), \quad E^{n} = -\phi_{,x}.
 \] 
  \STATE 2. On $\bigl[t^n, \, t^n+\frac{\Delta t}{2} \bigr]$ and for $\ell=0,1,2,3,4$ solve  
  VPFP with only the collision operator ($M^n \rightarrow 
  \widetilde{M}^{n}$):
\[
	\M{\ell,t} = \frac{1}{\varepsilon} \left\{ \ell (\ell-1) \M{\ell-2} 
   - \ell E^{n} \M{\ell-1} - \ell \M{\ell} \right\}.
\]
  \STATE 3. On $\bigl[t^n, \, t^n+{\Delta t}\bigr]$ solve the
  	{\it collisionless} quadrature-based moment-closure system \eqref{eqn:system} with 
	the appropriate definitions for $U$ and $F(U)$ ($\widetilde{M}^n \rightarrow 
  \widetilde{M}^{n+1}$).
 \STATE 4. Solve the Poisson equation:
 \[
 	-\phi_{,x,x} = \rho_{0}(x) - \widetilde{\rho}^{\, n+1}(x), \quad \widetilde{E}^{n+1} = -\phi_{,x}.
 \] 
\STATE 5. On $\bigl[t^n+\frac{\Delta t}{2}, \, t^n+{\Delta t}\bigr]$ and for $\ell=0,1,2,3,4$ solve 
VPFP with only the collision operator
($\widetilde{M}^{n+1} \rightarrow {M}^{n+1}$):
\[
	\M{\ell,t} = \frac{1}{\varepsilon} \left\{ \ell (\ell-1) \M{\ell-2} 
   - \ell \widetilde{E}^{n+1} \M{\ell-1} - \ell \M{\ell} \right\}.
\]
\end{algorithmic}

The spatial discretizations are handled via a high-order discontinuous Galerkin
discretization, which is briefly described below.

\subsection{High-order discontinuous Galerkin spatial discretization}
\label{sec:dg-method}
We make use of the discontinuous Galerkin (DG) method as developed by
Cockburn and Shu \cite{article:CoShu01} and implemented in the
{\tt DoGPack} software package \cite{dogpack} to solve hyperbolic
conservations of the form \eqref{eqn:system}. 

We begin by constructing an equally spaced numerical grid on $[a,b]$ consisting of $M$ elements,
each of the form:
%\begin{equation}
 $\Tm_i =  \left[x_i - \frac{\Delta x}{2}, x_i + \frac{\Delta x}{2} \right]$,
%\end{equation}
where $\Delta x = (b-a)/M$ is the grid spacing. 
Note that $x_{i} $ denotes the center of element $\Tm_i$.
Next we define the {\it broken} finite element space
\begin{equation}
    {\mathcal V}^{\Delta x} = \left\{ v^{\Delta x} \in L^{\infty}(\Omega): \,
    v^{\Delta x} |_{\Tm_i} \in P^k \, \forall i \right\},
\end{equation}
meaning that on each element $\Tm_i$, $v^{\Delta x}$ will
be a polynomial of degree at most $k$. 
The solution, $U^{\Delta x} \in {\mathcal V}^{\Delta x}$, restricted to element $\Tm_i$
can be written as
\begin{equation}
\label{eqn:Uapprox}
    U^{\Delta x} \bigl|_{\Tm_i} = \sum_{\ell=1}^k U^{\ell}(t) \, \varphi_{\ell}(\xi),
\end{equation}
where on each element
$x = x_i + \xi \left( \Delta x/2 \right)$, 
and $\varphi(\xi)$ are the orthonormal Legendre polynomials:
\begin{equation}
	\varphi(\xi) = \left\{ 1, \, \sqrt{3} \, \xi, \, \frac{\sqrt{5}}{2} \left( 3 \xi^2 -1 \right), \ldots \right\}.
\end{equation}

In order to obtain the semi-discrete DG method we multiply  \eqref{eqn:system}
by $\varphi_{j}(\xi)$, integrate over a single element $\Tm_i$,
replace the exact $U$ by \eqref{eqn:Uapprox}, and integrate-by-parts in $x$.
After simplification, this results in the following 
set of coupled ordinary differential equations in time:
\begin{equation}
\frac{d}{dt} U^{j}_{i} = \frac{1}{\Delta x} \int_{-1}^{1} F(U) \varphi_{j,\xi} \, d\xi
- \frac{1}{\Delta x} \left[ \varphi_j(1) \Flux_{i+\frac{1}{2}} - \varphi_j(-1) \Flux_{i-\frac{1}{2}} \right],
\end{equation}
where $\Flux_{i-\frac{1}{2}}$ is the {\it numerical flux} at interface $x=x_{i-\frac{1}{2}}$,
which must be calculated from an (approximate) Riemann solver (see \S \ref{sec:riemann}
below). In order to time advance this semi-discrete scheme, we make use of the standard
third-order total variation diminishing Runge-Kutta (TVD-RK) as described in
Gottlieb and Shu \cite{article:GoShu98}. We make use of the moment-limiters described
in Krividonova \cite{article:Kriv07} to suppress unphysical oscillations when required.

Finally we note that the Poisson equation in the operator split scheme
described above is solved using a local discontinuous Galerkin scheme
that is described in detail in Rossmanith and Seal \cite{article:RossSeal11}.

\subsection{Kinetic-based Riemann solvers}
\label{sec:riemann}
One missing ingredient from the discussion of the discontinuous Galerkin
scheme in the previous section is
a description of how the numerical flux, $\Flux_{i-\frac{1}{2}}$, is computed.
Since we have the ability to reconstruct the distribution function 
$\overline{f}(t,x,v)$ for any $(t,x)$, we can use a kinetic flux-vector splitting
approach (see for example Mandal and Deshpande \cite{article:MaDe94}).
In kinetic flux-vector splitting we split the flux into right-going contributions
immediately to the left of interface $x_{i-\frac{1}{2}}$ and left-going
contributions immediately to the right of interface $x_{i-\frac{1}{2}}$:
\begin{equation}
{\mathcal F}_{i-\frac{1}{2}} = {\mathcal F}^{+}_{i-\frac{1}{2}} + {\mathcal F}^{-}_{i-\frac{1}{2}},
\end{equation}
where
\begin{gather}
  {\mathcal F}^{+}_{i-\frac{1}{2}} = \int_{0}^{\infty}  \overline{f}\left(t,x^{-}_{i-\frac{1}{2}},v \right) \, dv \quad \text{and}
  \quad
  {\mathcal F}^{-}_{i-\frac{1}{2}} = \int_{-\infty}^{0} \overline{f}\left(t,x^{+}_{i-\frac{1}{2}},v \right) \, dv.
\end{gather}

\subsection{Stiff source term solution and the asymptotic-preserving condition}
 The final missing part of the Strang operator split algorithm presented in \S \ref{sec:strang}
 is the solver for the collision operator. A big advantage of considering fluid solvers
 over kinetic solvers in the context of VPFP is that the diffusion operator in $v$
 becomes an ODE for the moments. In particular, in the Strang split
 approach detailed in  \S \ref{sec:strang}, the electric field, $E(t,x)$, is
 frozen in time during each of the collision operator steps, meaning that
 the resulting ODEs are linear constant coefficient equations that
 can easily be solved exactly.
 The full solution over a time step $\left[t^n, t^n + \Delta t \right]$ with initial 
 data $\left( \M{0}^{n}, \M{1}^{n}, \M{2}^{n}, \M{3}^{n}, \M{4}^{n} \right)$ is 
 \begin{align}
\M{0}^{n+1} &= \M{0}^{n}, \\
\M{1}^{n+1} &= Z \left( E \M{0}^{n}+ \M{1}^{n} \right) - E \M{0}^{n}, \\
\begin{split}
 \M{2}^{n+1} &= Z^2 \left( (E^2 - 1) \M{0}^{n}+ 2 E \M{1}^{n} + \M{2}^{n} \right)  \\ &-2 E Z \left(E \M{0}^{n} + \M{1}^{n} \right) 
 +\left(1+E^2\right)  \M{0}^{n},
 \end{split} \\
 \begin{split}
 \M{3}^{n+1} &= Z^3 \left(E \left(E^2 - 3\right) \M{0}^n + 3 \left(E^2 - 1\right) \M{1}^n + 3E \M{2}^n + \M{3}^n\right) 
 \\ &- 3E Z^2 \left( \left(E^2 - 1 \right) \M{0}^n + 2E \M{1}^n + \M{2}^n\right) \\ &+ 3 \left(E^2 + 1\right) Z \left(E \M{0}^n + \M{1}^n\right) - E \left(E^2 + 3 \right) \M{0}^n,
\end{split} \\
\begin{split}
\M{4}^{n+1} &= Z^4 \bigl( \left( E^4 - 6 E^2 + 3 \right) \M{0}^n + 4 E \left(E^2 - 3 \right) \M{1}^n + 
      6 \M{2}^n \left(E^2 - 1 \right) \\ &+ 4 E \M{3}^n + \M{4}^n \bigr) - 
   4 E Z^3 \left(E \left(E^2 - 3\right) \M{0}^n + 3 \M{1}^n \bigl(E^2 - 1\right) \\ &+ 3E \M{2}^n 
   + \M{3}^n \bigr) + 
   6\left(E^2 + 1\right) Z^2 \left(\left(E^2 - 1\right) \M{0}^n + 2E\M{1}^n + \M{2}^n \right) \\ &- 
   4E \left(E^2 + 3\right) Z \left( E \M{0}^n + \M{1}^n \right) + \left(E^4 + 6 E^2 + 3\right) \M{0}^n,
 \end{split}
 \end{align}
 where
 $E = E^{n}$ and $Z = \exp\left[-\Delta t/\varepsilon\right]$.
 
 \section{Numerical simulations in the high-field limit}
 \label{sec:numerical-examples}
In order to verify the proposed DG operator split method
using the quadrature-based moment-closure we consider
two test cases from Wang and Jin \cite{article:WaJi11}:
(1) verification of the asymptotic-preserving property and
(2) a periodic Riemann problem. All of these
problems are defined on $[0,1]$ with periodic boundary conditions.

\subsection{Verification of the asymptotic-preserving property}
In order to verify the asymptotic-preserving property of the
proposed scheme, we attempt two versions of the same problem
 from Wang and Jin \cite{article:WaJi11}.
 
For the first problem we start with an isothermal Gaussian:
\begin{align}
\label{eqn:wang_ic_f}
f(0,x,v) &= \frac{\rho(0,x)}{\sqrt{2\pi}} \exp\left[ -\frac{1}{2} \left( v + E(0,x) \right)^2 \right], \\
\rho(0,x) &= \frac{\sqrt{2\pi}}{2} \left( 2 + \cos\left(2\pi x\right) \right).
\end{align}
with a neutralizing background charge of
\begin{equation}
 \rho_{0}(x) =  \frac{\sqrt{2\pi}}{1.2660658777520083} \exp\left[\cos(2\pi x) \right].
\end{equation}
The quantity $\| M_{1} + \rho E \|_{L^2}$ is plotted for various $\varepsilon$
as a function of time in Figure \ref{fig:ap}(a). These results verify that
$\| M_{1} + \rho E \|_{L^2} = {\mathcal O}\left( \varepsilon \right)$ for all $t$.

For the second problem we start with initial data that is not in isothermal
equilibrium:
\begin{equation}
f(0,x,v) = \frac{\rho(0,x)}{2\sqrt{2\pi}}\left( \exp\left[ -\frac{1}{2} \left( v + 1.5 \right)^2 \right]
+ \exp\left[ -\frac{1}{2} \left( v - 1.5 \right)^2 \right] \right),
\end{equation}
with the same $\rho(0,x)$ and neutralizing background charge 
as in the previous example.
The quantity $\| M_{1} + \rho E \|_{L^2}$ is plotted for various $\varepsilon$
as a function of time in Figure \ref{fig:ap}(b). These results verify that
the numerical schemes immediately drives the non-equilibrium initial
data near the equilibrium distribution such that
$\| M_{1} + \rho E \|_{L^2} = {\mathcal O}\left( \varepsilon \right)$.

The simulations presented in Figure \ref{fig:ap} were done with
the bi-Gaussian moment-closure. The bi-delta and bi-B-spline
methods give near identical results for this problem.

\begin{figure}
\begin{center}
\begin{tabular}{c}
   (a)\includegraphics[width=70mm]{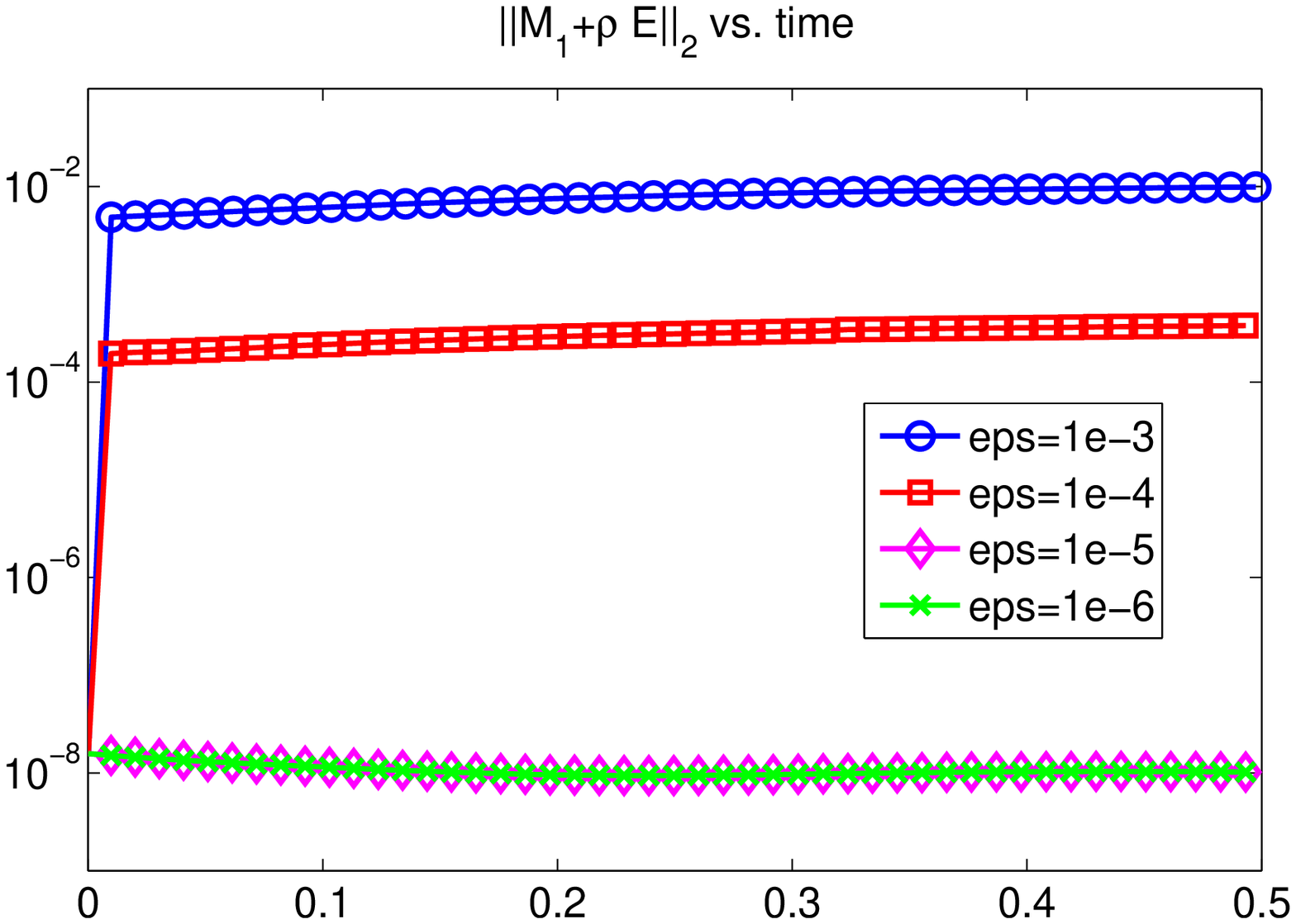} \\
   (b)\includegraphics[width=70mm]{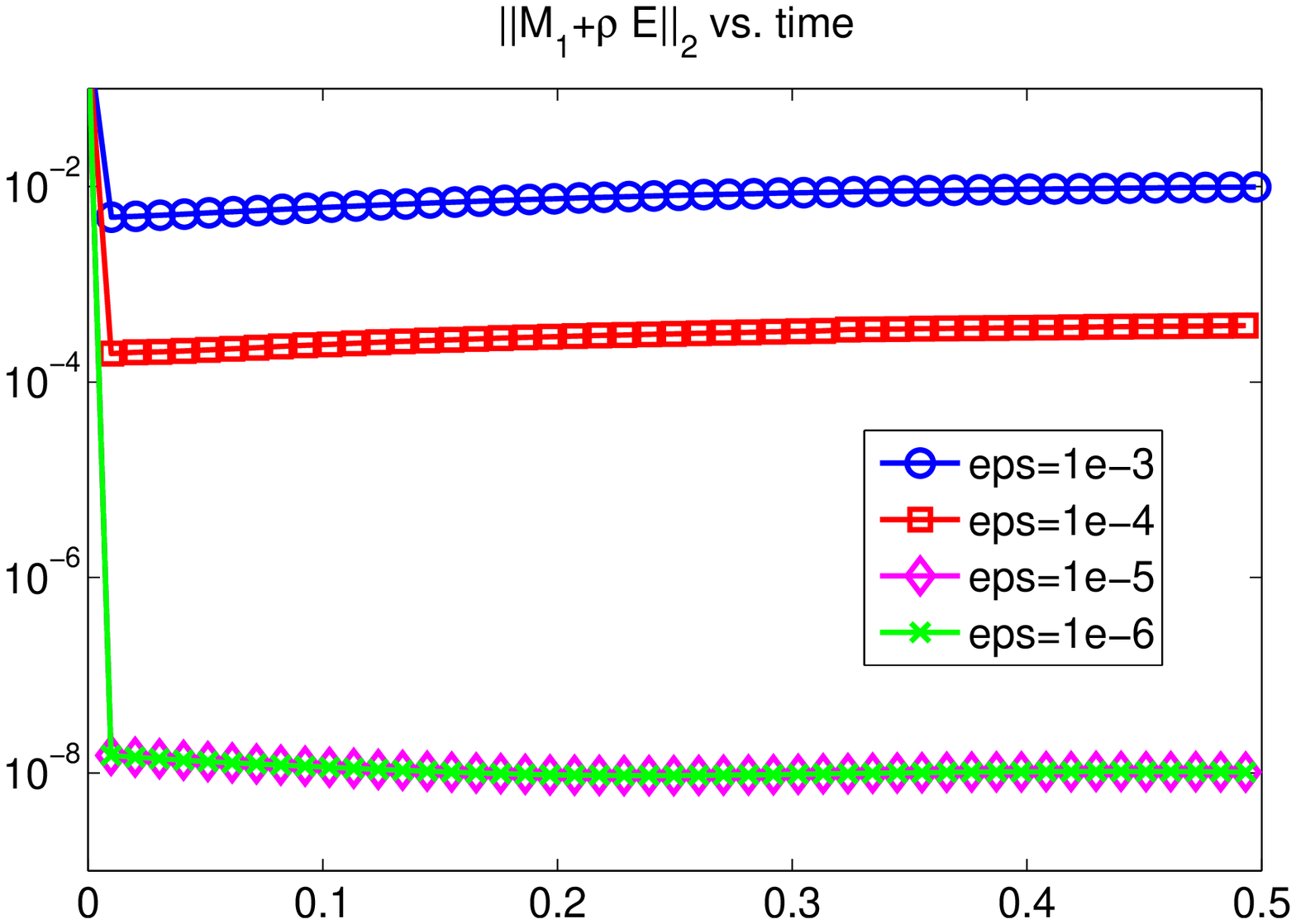}
\end{tabular}
  \caption{Verification of the asymptotic preserving capability
  of the method. Shown in these two plots are the $L^2$-norms
  of the difference in the fluid mean velocity ($u$) and the equilibrium
  mean velocity ($-\rho E$) as a function of time. 
  In Panel (a) the initial conditions are already
  in equilibrium, while in Panel (b) the initial conditions is
  not in equilibrium. We show only every 20th time step
  value so that the various curves are more easily identified.
  We note that for $\varepsilon=10^{-5}$ and $\varepsilon=10^{-6}$,
  the resolution in the simulations is not enough to resolve
  the non-equilibrium deviations; and therefore, the difference
  between $u$ and $-\rho E$ is negligible. All runs were done
  with 64 mesh elements. \label{fig:ap}}
\end{center}
\end{figure}

\subsection{Double periodic Riemann problem}
The initial data is the distribution function \eqref{eqn:wang_ic_f}
with
\begin{equation}
\left( \rho(0,x), \, \rho_0(x) \right) = 
  \begin{cases}
  	\left( 1/8, \, 1/2 \right) & \text{if} \, \, 0 \le x < 1/4, \\
	\left( 1/2, \, 1/8 \right) & \text{if} \, \, 1/4 \le x < 3/4, \\
	\left( 1/8, \, 1/2 \right) & \text{if} \, \, 3/4 \le x \le 1.
  \end{cases}	
\end{equation}
The solution using the bi-B-spline moment-closure 
is shown in Figure \ref{fig:priemann}. The results agree well with those in
Wang and Jin \cite{article:WaJi11}. The bi-Gaussian moment-closure
has difficulties with this problem due to the steep gradients in the
solution in regions where $\alpha$ is small but non-zero. 

\begin{figure}
\begin{center}
\begin{tabular}{c}
   (a)\includegraphics[width=70mm]{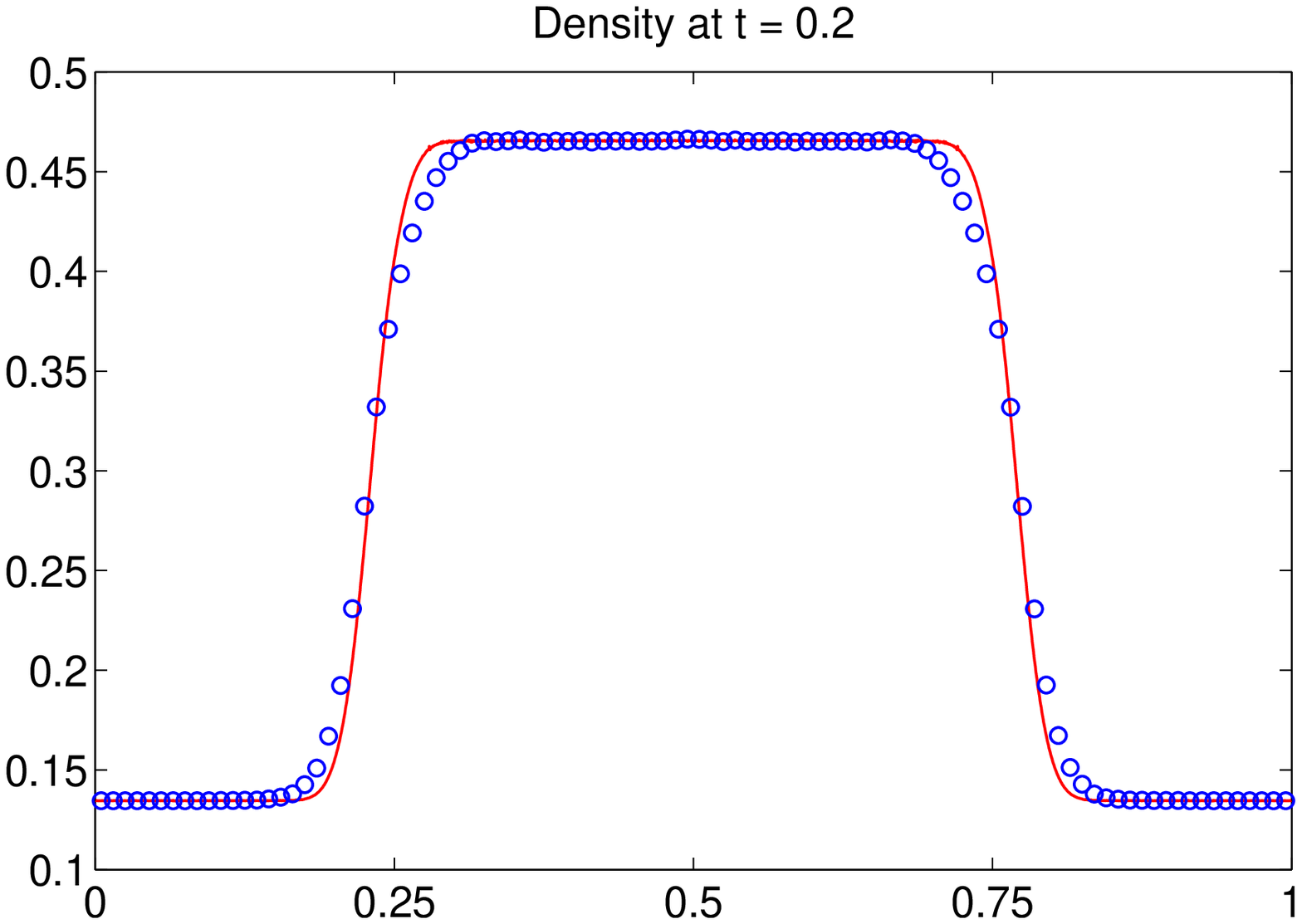}\\
   (b)\includegraphics[width=70mm]{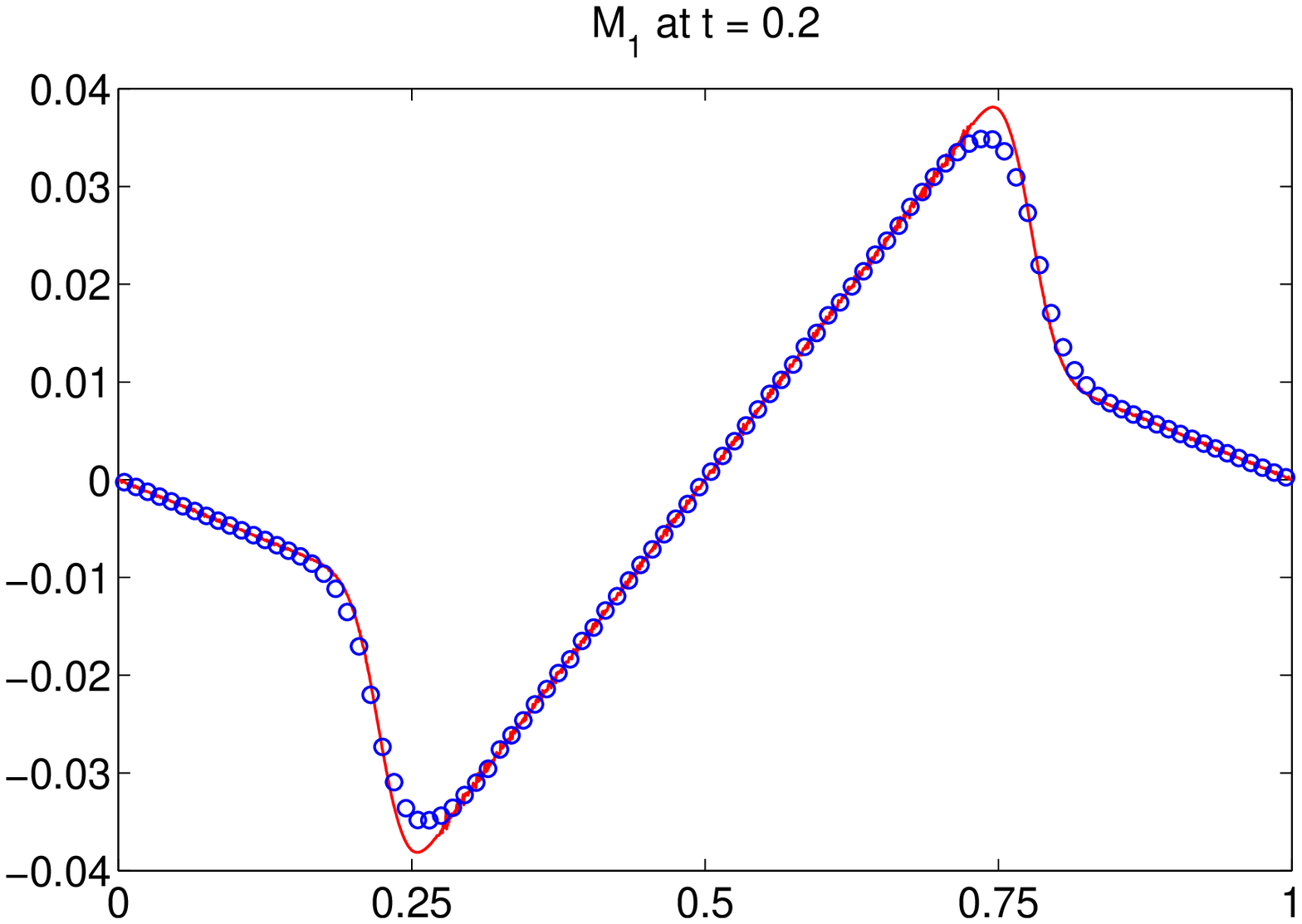}
\end{tabular}
  \caption{Periodic Riemann problem of Wang and Jin \cite{article:WaJi11}.
  Shown are the solutions with 100 elements (blue circles) and 2000 elements
  (red line). This simulation is difficult for the bi-Gaussian moment-closure
  due to the steep gradients in the solution in regions where $\alpha$ is
  small but non-zero. These simulations were instead done with the bi-B-spline
  moment-closure and the results show good agreement with the results of
  Wang and Jin \cite{article:WaJi11}. \label{fig:priemann}}
\end{center}
\end{figure}
 
 \section{Summary}
 In this work we considered quadrature-based moment-closure
 methods using two quadrature points. We briefly investigated
 the properties of these methods and showed connections
 between bi-delta, bi-Gaussian, and bi-B-spline quadrature
 methods. We then applied this formulation to the Vlasov-Poisson-Fokker-Planck
 system in the high-field limit, and, using a high-order discontinuous Galerkin
 scheme with Strang operator splitting, verified the scheme on 
 two test problems. 
 Future work will focus on multidimensional plasma physics
 applications.
 
 \bigskip

\noindent
{\bf Acknowledgements.}
This work was supported in part by NSF grant DMS--1016202.

\end{document}